\documentclass[11 pt]{amsart}
\usepackage{amssymb}

\usepackage{bm}

\makeatletter
\DeclareRobustCommand*{\bfseries}{%
  \not@math@alphabet\bfseries\mathbf
  \fontseries\bfdefault\selectfont
  \boldmath
}
\makeatother

\calclayout

\def\ignore#1{\relax}

\newcommand*\refc{\eqref}    

\usepackage[svgnames]{xcolor}
\usepackage[plainpages=false, pdfpagelabels,  colorlinks=true, pdfstartview=FitV, linkcolor=DarkBlue, citecolor=DarkRed, urlcolor=blue]{hyperref}

\usepackage{graphicx}

\newcommand\inlinegraphic[2][{scale=1.0}]{\begin{array}{c} \includegraphics[#1]{./EPS/#2}\end{array}}

\numberwithin{equation}{section}
\numberwithin{figure}{section}

\def\Z{{\mathbb Z}}

\def\C{{\mathbb C}}

\def\la{\lambda}

\def\inv{^{-1}}

\newcommand{\m}[5]{m^{#1}_{({#2}, {#3}), \, ({#4}, {#5})}}

\newcommand\mm[3]{m^{#1}_{#2, \, #3}}

\newcommand{\leftexp}[2]{{\vphantom{#2}}^{#1}{#2}}

\newcommand\paren[1]{(#1)}

\def\blam{{\bm\lambda}}
\def\bmu{{\bm\mu}}
\def\bnu{{\bm \nu}}
\def\bgamma{{\bm \gamma}}

\def\symn{\mathfrak S_n}
\def\sym{\mathfrak S}

 \def\GammaDominates{\unrhd_{_\Gamma}}
 \def\StrictlyGammaDominates{\rhd_{_\Gamma}}
\def\multicomp{C_n^\Gamma}
 \def\multipart{\Lambda_n^\Gamma}
\def\power #1{^{(#1)}}
\def\dominates{\unrhd}
\def\StrictlyDominates{\rhd}

\def\p #1{ \bm {#1}}
\def\pbar #1{\overline{\p {#1}}}
\newcommand\generator[2]{\delta_{#1}^{#2}}

\def\mft{\mathfrak t}
\def\mfs{\mathfrak s}
\def\mfu{\mathfrak u}
\def\mfv{\mathfrak v}

\def\Hom{{\rm Hom}}
\def\End{{\rm End}}
 \def\rank{{\rm rank}}
 \def\Ind{{\rm Ind}}
 \def\tr{{\rm tr}}
 \def\Tr{{\rm Tr}}
 \def\spn{{\rm span}}
\def\cl{{\rm cl}}
\def\comp{{\rm comp}}
\def\char{{\rm char}}

\theoremstyle{plain}
\newtheorem{theorem}{Theorem}[section]

\theoremstyle{plain}
\newtheorem{proposition}[theorem]{Proposition}

\theoremstyle{plain}
\newtheorem{corollary}[theorem]{Corollary}

\theoremstyle{plain}
\newtheorem{lemma}[theorem]{Lemma}

\theoremstyle{plain}
\newtheorem{claim}[theorem]{Claim}

\theoremstyle{definition}
\newtheorem{definition}[theorem]{Definition}

\theoremstyle{definition}
\newtheorem{example}[theorem]{Example}

\theoremstyle{definition}
\newtheorem{remark}[theorem]{Remark}

\theoremstyle{definition}
\newtheorem{notation}[theorem]{Notation}

\theoremstyle{definition}

\theoremstyle{definition}

\theoremstyle{plain}

\title{Cellularity of Wreath Product Algebras and $A$--Brauer algebras}

\author{T. Geetha}
\address{The Institute of Mathematical Sciences\\Chennai, India}
\email{ geetha\_curie@yahoo.co.in}

\author{Frederick M. Goodman}
\address{Department of Mathematics\\ University of Iowa\\ Iowa
City, Iowa, USA}
\email{ frederick-goodman@uiowa.edu} 

\dedicatory{Dedicated to V.S. Sunder on the occasion of his 60th birthday.}
\thanks{We thank Andrew Mathas for suggesting that the method of ~\cite{DJM} could be adapted
to prove the cellularity of wreath products.  T. Geetha was supported by a postdoctoral fellowship at the Institute of Mathematical Sciences, Chennai, India. }

\begin{document}
 
  \begin{abstract}  A cellular algebra is called {\em cyclic cellular} if all cell modules are cyclic.  Most important examples of cellular algebras appearing in representation theory are in fact cyclic cellular.
 We prove that if $A$ is a cyclic cellular algebra, then the wreath product algebras $A\wr \symn$ are also cyclic cellular.  We also introduce $A$--Brauer algebras, for algebras $A$ with an involution and trace. This class of algebras includes, in particular, $G$--Brauer algebras for non-abelian groups $G$.  
We  prove that if $A$ is cyclic cellular then the $A$--Brauer algebras $D_n(A)$ are also cyclic cellular.
 \end{abstract}
 \maketitle

 \section{Introduction}\label{introduction}
 The concept of cellularity of algebras was introduced by Graham and Lehrer ~\cite{Graham-Lehrer-cellular}.  Cellularity is useful for analyzing the representation theory of important classes of algebras such as Hecke algebras, $q$--Schur algebras, Brauer and BMW algebras, etc.
 
 In this paper, we introduce a variant of the notion of cellularity:  A cellular algebra $A$  is called {\em cyclic cellular} if all of its cell modules are cyclic $A$--modules.  Although cyclic cellularity is nominally stronger than cellularity,  most important classes of cellular algebras appearing in representation theory are in fact cyclic cellular. For example, the algebras mentioned above -- Hecke algebras of type $A$, $q$--Schur algebras, Brauer algebras and BMW algebras -- are cyclic cellular.
 
 It appears that the requirement of cyclic cell modules eliminates some potential pathology allowed by the general notion of cellularity; for example, an abelian cellular algebra is cyclic cellular if and only if all the cell modules have rank 1.
 
 Our main theorem regarding cyclic cellular algebras is the following:  if $A$ is a cyclic cellular algebra, then the wreath product of $A$ with the symmetric group $\symn$ is also cyclic cellular.
 
 In the second part of the paper, we introduce a new class of algebras, the $A$--Brauer algebras, which are a sort of wreath product of an algebra $A$ with the Brauer algebra.  The definition of the $A$--Brauer algebras requires that $A$ be an algebra with involution $*$ and with a $*$--invariant trace.  The class of $A$--Brauer algebras includes in particular $G$--Brauer algebras for non-abelian groups $G$, generalizing the construction for abelian groups $G$ by Parvathi and Savithri  ~\cite{G-Brauer}.
 
 Our second main theorem is that if $A$ is a cyclic cellular algebra with an involution--invariant trace, then the $A$--Brauer algebras $D_n(A)$ for $n \ge 1$ are also cyclic cellular.

\section{Cellular algebras}
\subsection{Definition of cellularity}
We recall the definition of {\em cellularity}  from  Graham and Lehrer ~\cite{Graham-Lehrer-cellular}.   
The version given here is slightly weaker than the original definition in 
~\cite{Graham-Lehrer-cellular}.  The advantages of the weaker definition are discussed in 
~\cite{Goodman-Graber1, Goodman-Graber2}.

\begin{definition}[\cite{Graham-Lehrer-cellular}]  \label{weak cellularity}  Let $R$ be 
an integral domain and $A$ a unital $R$--algebra.  A {\em cell datum} for $A$ 
consists of  an $R$--linear algebra involution  $a \mapsto a^*$;  a finite partially 
ordered set $(\Gamma, \ge)$;  for each $\gamma \in \Gamma$  a finite index set 
$\mathcal T(\gamma)$;   and  a subset 
$$
\mathcal C = \{ c_{s, t}^\gamma :  \gamma \in \Gamma \text{ and }  
s, t \in \mathcal T(\gamma)\} \subseteq A
$$ 
with the following properties:
\begin{enumerate}
\item  $\mathcal C$ is an $R$--basis of $A$.
\item   \label{mult rule} For each $\gamma \in \Gamma$,  let $\bar A^\gamma$  
be the span of the  $c_{s, t}^\mu$  with $\mu > \gamma$.   Given $\gamma \in 
\Gamma$,  $s \in \mathcal T(\gamma)$, and $a \in A$,   there  exist coefficients  
$r_v^s( a) \in R$ such that for all $t \in \mathcal T(\gamma)$:
$$
a c_{s, t}^\gamma  \equiv \sum_v r_v^s(a)  c_{v, t}^\gamma  \mod  \bar A^\gamma.
$$
\item  $(c_{s, t}^\gamma)^* \equiv c_{t, s}^\gamma \mod \bar A^\gamma$ for all 
$\gamma\in \Gamma$ and, $s, t \in \mathcal T(\gamma)$.
\end{enumerate}
For brevity,  we will write that  $(\mathcal C, \Gamma)$ is a  cellular basis of $A$.   When we need to  be more explicit, we write that $(A, *, \Gamma, \ge, \mathcal T, \mathcal C)$ is a cell datum.
\end{definition}

Note that it follows from Definition \ref{weak cellularity} (2) and (3) that 
\begin{equation} \label{equation: right multiplication rule}
c_{s, t}^\gamma a  \equiv \sum_w r_w^t(a^*)  c_{s, w}^\gamma  \mod  \bar A^\gamma.
\end{equation}

The original definition of Graham and Lehrer  includes a stronger version of  condition (3), as follows:

\begin{definition} \label{strict cellularity}
An algebra $A$ with algebra involution $*$ over an integral domain $R$ is said to 
be {\em strictly cellular} if it satisfies the conditions of Definition \ref{weak 
cellularity}, except that condition (3) is replaced by the stronger requirement:

\smallskip
\noindent
\quad  \ \ (3$'$) \   $(c_{s, t}^\gamma)^* = c_{t, s}^\gamma$ for all $\gamma\in \Gamma$ and 
$s, t \in \mathcal T(\gamma)$.
\end{definition}

  All the statements and arguments of ~\cite{Graham-Lehrer-cellular} remain valid 
with the weaker Definition \ref{weak cellularity}.  The two definitions are equivalent 
in case $2$ is invertible in the ground ring $R$.   
Basis free characterizations of cellularity have been given in ~\cite{KX-Morita, KX-structure, Goodman-Graber1}.
 
\subsection{Structures related to cellularity}
We recall some basic structures related to cellularity, see ~\cite{Graham-Lehrer-cellular}.
Let $A$ be  a cellular algebra with cell datum $(A, *, \Gamma, \ge, \mathcal T, \mathcal C)$.
Given $\gamma\in\Gamma$,   let $A^\gamma$ denote the span of the $c_{ s, t}^{\mu}$ with 
$\mu  \geq \gamma$.  It follows that both $A^\gamma$ and $\bar A^\gamma$ 
(defined above) are $*$--invariant two sided ideals of $A$.

For $\gamma \in \Gamma$,  the {\em left  cell module} $\Delta^\gamma$ 
is defined as follows: as   an $R$--module, $\Delta^\gamma$  is free with basis indexed by $
\mathcal T(\gamma)$, say  $\{c_{ s}^\gamma$ : $ s \in \mathcal T(\gamma)\}$;   for each $a \in A$, 
the action of $a$ on $\Delta^\gamma$ is defined by  
$ ac_{ s}^\gamma=\sum_{ v} r_{ v}^{ s}(a)   c_{ v}^\gamma$ 
where the elements $r_{ v}^{ s}(a) \in R$ are the coefficients  in Definition \ref{weak cellularity} (\ref{mult rule}).   Note that for any $t \in \mathcal T(\gamma)$, the assignment $c_s^\gamma \mapsto c^\gamma_{s, t} + \bar A^\gamma$
defines an injective $A$--module homomorphism from $\Delta^\gamma$ to $A^\gamma/\bar A^\gamma$.  

If $A$ and $B$ are $R$--algebras with involutions denoted by $*$,  then we have a functor
$M \mapsto M^*$   from
$A$--$B$ bimodules to $B$--$A$ bimodules, as follows.  As an $R$--module, $M^*$ is just a copy of $M$ with elements marked by $*$.  The $B$--$A$ bimodule structure of $M^*$ is determined by 
$b x^* a = (a^* x b^*)^*$.    We have a natural isomorphism $M^{**}  \cong M$, via $x^{**} \mapsto x$.   In particular, taking $B$ to be $R$ with the trivial involution,  we get a functor from left $A$--modules to right $A$ modules.  Similarly, we get  a functor from right $A$--modules to left $A$--modules.   (If $\Delta \subset A$ is a left or right ideal, we have two meanings for $\Delta^*$, namely application of the functor $*$, or application of the involution in $A$,  but these agree.)
If ${}_A M$ is a left $A$--module and  $N_A$ is a right  $A$--module,  then
$$
(M \otimes_R N)^* \cong  N^* \otimes_R M^*,  
$$
as $A$--$A$ bimodules, with the isomorphism determined by $(m \otimes n)^* \mapsto n^* \otimes m^*$. 
In particular if ${}_A M$ is a left $A$--module and we identify $M^{**}$ with $M$  and
$(M \otimes M^*)^*$ with $M^{**} \otimes M^*  = M \otimes M^*$,   then we have $(x \otimes y^*)^* = y \otimes x^*$.    

Now we apply these observations with $A$ a cellular algebra and $\Delta^\gamma$ a cell module.  The assignment $$\alpha_\gamma : c^\gamma_{s, t}  +  \bar A^\gamma  \mapsto  c^\gamma_s \otimes (c^\gamma_t)^*$$ determines an
$A$--$A$ bimodule isomorphism from $A^\gamma/\bar A^\gamma$ to $\Delta^\gamma \otimes_R (\Delta^\gamma)^*$. Moreover,
we have $*\circ \alpha_\gamma =  \alpha_\gamma \circ *$, which reflects the cellular algebra axiom
$(c^\gamma_{s, t})^* \equiv c^\gamma_{t, s} \mod \bar A^\gamma$.   The importance of the maps $\alpha_\gamma$ for the structure of cellular algebras was stressed by K\"onig and Xi 
~\cite{KX-Morita, KX-structure}.

There exists a symmetric $R$--valued bilinear form 
$\langle\, , \rangle$ on $\Delta^\gamma$ such that for all $s, t, u, v \in \mathcal T(\gamma)$,
$$
c_{s, t}^\gamma c_{u, v}^\gamma \equiv \langle c_t^\gamma, c_u^\gamma \rangle
c_{s, v}^\gamma \mod \bar A^\gamma.
$$
This bilinear form plays an essential role in the theory of cellular algebras, see ~\cite{Graham-Lehrer-cellular}.

\subsection{Equivalent cellular bases}
A cellular algebra $A$ always admits many different cellular basis.  In fact, any choice of an $R$--basis in each cell module can be globalized to a cellular basis of $A$, see Lemma \ref{lemma: globalizing bases of cell modules} below.  

Let $A$ be a cellular algebra with cell datum $(A, *, \Gamma, \ge, \mathcal T, \mathcal C)$.
We say that a cellular basis
$$
\mathcal B = \{b^\gamma_{s, t} : \gamma \in \Gamma \text{ and }  s, t \in \mathcal T(\gamma)\}
$$
is {\em equivalent} to the original cellular basis $\mathcal C$ if it determines the same ideals
$A^\gamma$ and the same cell modules as does $\mathcal C$.   More precisely, the requirement is that
\begin{enumerate}
\item for all $\gamma \in \Gamma$,
$$
A^\gamma = \spn\{b^{\gamma'}_{s, t} : \gamma' \ge \gamma \text{ and } s, t \in \mathcal T(\gamma')\} \text{, and}
$$
\item for all $\gamma \in \Gamma$ and $t \in \mathcal T(\gamma)$, 
$$
\spn\{b^\gamma_{s, t} + \bar A^\gamma : s \in   \mathcal T(\gamma) \} \cong \Delta^\gamma,
$$
as $A$--modules.
\end{enumerate}

\begin{lemma} \label{delta tensor b isomorphic to delta}
Let $A$ be a cellular algebra with cell datum $(A, *, \Gamma, \ge, \mathcal T, \mathcal C)$.
Let $\gamma \in \Gamma$ and let $b \in \Delta^\gamma$ be non--zero.  Then $x \mapsto x \otimes b^*$ is an $A$--module isomorphism of $\Delta^\gamma$ onto $\Delta^\gamma \otimes b^* \subseteq \Delta^\gamma \otimes_R (\Delta^\gamma)^*$. 
\end{lemma}

\begin{proof} Since $(\Delta^\gamma)^*$ is a free $R$--module, it is torsion free; hence
$$
\Delta^\gamma \cong \Delta^\gamma \otimes_R R \cong \Delta^\gamma \otimes_R R b^* = 
 \Delta^\gamma \otimes b^*
$$
as $A$--modules. 
Explicitly, the isomorphism is $x \mapsto x \otimes b^*$. 
\end{proof}

\begin{lemma}  \label{lemma: globalizing bases of cell modules}
Let $A$ be a cellular algebra with cell datum $(A, *, \Gamma, \ge, \mathcal T, \mathcal C)$.  For each $\gamma \in \Gamma$,  fix an $A$--$A$ bimodule isomorphism  $\beta_\gamma : A^\gamma/\bar A^\gamma \to
\Delta^\gamma \otimes_R (\Delta^\gamma)^*$ satisfying $*\circ \beta_\gamma = \beta_\gamma \circ *$, 
and let $\{b_t : t \in \mathcal T(\gamma)\}$ be an $R$--basis of $\Delta^\gamma$.   Finally, for each $\gamma \in \Gamma$ and each pair $s, t \in \mathcal T(\gamma)$,  let $b^\gamma_{s, t}$ be an arbitrary lifting of $\beta_\gamma\inv(b_s \otimes b_t^*)$ in $A^\gamma$. 
Then
$$
\mathcal B = \{b^\gamma_{s, t} : \gamma \in \Gamma \text{ and }  s, t \in \mathcal T(\gamma)\}
$$
is a cellular basis of $A$ equivalent to the original cellular basis $\mathcal C$. 
\end{lemma}

\begin{proof} It follows from ~\cite{Goodman-Graber2}, Lemma 2.3, that $\mathcal B$ is a cellular basis.
It remains to show that $\mathcal B$ is equivalent to the original cellular basis $\mathcal C$. For each fixed $\gamma \in \Gamma$,  $\{ b_s \otimes b_t^*:  s, t \in \mathcal T(\gamma)\}$ is an $R$--basis of $\Delta^\gamma \otimes_R (\Delta^\gamma)^*$.  It follows that
$$
\spn\{b^\gamma_{s, t} : s, t \in \mathcal T(\gamma)\} + \bar A^\gamma= A^\gamma.
$$
By induction on the partial order of $\Gamma$, we obtain that
$$
A^\gamma = \spn\{b^{\gamma'}_{s, t} :  \gamma' \ge \gamma \text{ and } s, t \in \mathcal T(\gamma')\}.
$$
By definition of $b^\gamma_{s, t}$, we have that $\beta_\gamma(b^\gamma_{s, t} + \bar A^\gamma) = b_s \otimes b_t^*$.  Thus for fixed $t \in \mathcal T(\gamma)$, the restriction of $\beta_\gamma$ is an isomorphism from $\spn\{b^\gamma_{s, t} + \bar A^\gamma : s \in \mathcal T(\gamma) \}$ onto $\Delta^\gamma \otimes b_t^*$.  By Lemma \ref{delta tensor b isomorphic to delta}, 
$\Delta^\gamma \otimes b_t^* \cong \Delta^\gamma$ as $A$--modules.  Thus we have
$$
\spn\{b^\gamma_{s, t} + \bar A^\gamma : s \in \mathcal T(\gamma) \} \cong \Delta^\gamma.
$$

\end{proof}

\subsection{Cyclic cellular algebras}

\begin{definition}   A cellular algebra is said to be {\em cyclic cellular} if every cell module of $A$ is cyclic.
\end{definition}

We will also say that a cellular basis is cyclic cellular if the cell modules defined via this basis are cyclic.

\begin{lemma} \label{lemma: equivalent conditions for cyclic cellular algebra}
 Let $A$ be a   cellular   algebra over an integral domain $R$ with cell datum \break $(A, *, \Gamma, \ge, \mathcal T, \mathcal C)$.  The following are equivalent:
\begin{enumerate}
\item  $A$ is cyclic cellular.
\item  For each $\gamma \in \Gamma$,  there exists an element $y_\gamma \in A^{\gamma}$ with the properties:
\begin{enumerate}
\item $y_\gamma \equiv y_\gamma^* \mod{\bar A^{\gamma}}$.
\item $A^{\gamma} =  A y_\gamma A +  \bar A^{\gamma}$.
\item    $(A y_\gamma + \bar A^{\gamma})/\bar A^{\gamma} \cong  \Delta^\gamma$, as $A$--modules. 
\end{enumerate}
\end{enumerate}
\end{lemma}

\begin{proof}  Suppose that $A$ is cyclic cellular.  For each $\gamma \in \Gamma$,  chose a generator
$\generator {} \gamma$ of the cell module  $\Delta^\gamma$.  Let $y_\gamma$ be any lifting in $A^{\gamma}$ of
$\alpha_\gamma\inv(\generator {} \gamma \otimes (\generator {} \gamma)^*)$.  
 Then property (2a) holds because
$(\generator {} \gamma \otimes (\generator {} \gamma)^*)^* = \generator {} \gamma \otimes (\generator {} \gamma)^*$.
 Property (2b) holds because
$A(\generator {} \gamma \otimes (\generator {} \gamma)^*)A = \Delta^\gamma \otimes_R  (\Delta^\gamma)^*$.  

The restriction of $\alpha_\gamma$ yields an $A$--module isomorphism $a y_\gamma + \bar A^\gamma
\mapsto a \generator {} \gamma \otimes (\generator {} \gamma)^*$.  By Lemma \ref{delta tensor b isomorphic to delta}, 
$x \otimes (\generator {} \gamma)^* \mapsto x$  is an $A$--module isomorphism from $\Delta^\gamma \otimes  (\generator {} \gamma)^*$ onto $\Delta^\gamma$.  By composing these isomorphisms, we obtain the isomorphism $a y_\gamma + \bar A^\gamma \mapsto a \generator {} \gamma$ from
$(A y_\gamma + \bar A^{\gamma})/\bar A^{\gamma} $ to $\Delta^\gamma$.  Thus property (2c) holds.

Conversely, if (2) holds, then in particular (2c) implies that each cell module is cyclic.
\end{proof}

For the remainder of this section, let $A$ be a cyclic cellular algebra with cell datum  $(A, *, \Gamma, \ge, \mathcal T, \mathcal C)$, , and write 
 $\mathcal C = \{ c_{s, t}^\gamma :  \gamma \in \Gamma \text{ and }  s, t \in \mathcal T(\gamma)\}$.

\begin{notation}\label{notation cyclic cellular}
 For each $\gamma \in \Gamma$,  let $\generator {} \gamma$ be a generator of the cell module $\Delta^\gamma$, and let $y_\gamma$ be  a lifting in $A^\gamma$ of $\alpha_\gamma\inv(\generator {} \gamma \otimes (\generator {} \gamma)^*)$.   Let $\{c^\gamma_\mft : \mft \in \mathcal T(\gamma)\}$ be the standard basis of the cell module $\Delta^\gamma$ derived from the cellular basis $\mathcal C$ of $A$.   Since $\Delta^\gamma$ is cyclic, there exist elements $v_\mft \in A$ such that
$c^\gamma_\mft = v_\mft \generator {} \gamma$.    We denote
$$
V^\gamma = \{v_\mft : \mft \in \mathcal T(\gamma)\}.
$$
\end{notation}

\begin{lemma}  \label{lemma span of v y gamma} \mbox{}
\begin{enumerate}
\item  $Ay_\gamma + \bar A^\gamma = \spn\{ v y_\gamma: v \in V^\gamma\} + \bar A^\gamma$.
\item  $y_\gamma A + \bar A^\gamma = \spn\{  y_\gamma v^*: v \in V^\gamma\} + \bar A^\gamma$.
\end{enumerate}
\end{lemma}

\begin{proof} Considering the isomorphism  $a y_\gamma + \bar A^\gamma \mapsto a \generator {} \gamma$ from $(A y_\gamma + \bar A ^\gamma)/\bar A^\gamma$ to $\Delta^\gamma$, we see that
 $\{v y_\gamma + \bar A^\gamma : v \in V^\gamma\}$ is a basis of $(A y_\gamma + \bar A ^\gamma)/\bar A^\gamma$.  Lifting to $A^\gamma$, we have that  
 $\spn\{ v y_\gamma: v \in V^\gamma\} + \bar A^\gamma = Ay_\gamma + \bar A^\gamma$. 
 Applying the involution $*$ and using that $y_\gamma^* \equiv y_\gamma \mod \bar A^\gamma$, 
 we obtain statement (2).  
 \end{proof}

\begin{lemma}  \mbox{}
\begin{enumerate}
\item
For $\gamma \in \Gamma$ and $\mfs, \mft \in \mathcal T(\gamma)$, we have
$$
v_\mfs y_\gamma v_\mft^* \equiv c^\gamma_{\mfs, \mft} \mod {\bar A^\gamma}.
$$
\item  $\{ v_\mfs y_\gamma v_\mft^* : \gamma \in \Gamma \text{ and }
\mfs, \mft \in \mathcal T(\gamma)\}$ is a cellular basis  of $A$ equivalent to the original cellular basis $\mathcal C$ of $A$.  
\end{enumerate}
\end{lemma}

\begin{proof} Point (1) holds because both $v_\mfs y_\gamma v_\mft^*$ and
$c^\gamma_{\mfs, \mft}$  are liftings in $A^\gamma$ of $\alpha_\gamma\inv(c^\gamma_\mfs \otimes  (c^\gamma_\mft )^*)$.  Point (2) follows from point (1).
\end{proof}

We record a version of Lemma \ref{lemma: globalizing bases of cell modules} that is adapted to the context of cyclic cellular algebras.

\begin{lemma}  For each $\gamma \in \Gamma$, let
 $\{b_\mft: \mft \in \mathcal T(\gamma)\}$ be an $R$--basis of the cell module $\Delta^\gamma$.   For $\mft  \in \mathcal T(\gamma)$,  choose
$v'_\mft \in A$ such that $b_\mft = v'_\mft \,\generator {} \gamma$.    For $\mfs, \mft \in \mathcal T(\gamma)$,  let $b^\gamma_{\mfs, \mft} = v'_\mfs y_\gamma ( v'_\mft)^*$
Then
$\mathcal B = \{ b^\gamma_{\mfs, \mft} : \gamma \in \Gamma \text{ and } \mfs, \mft \in \mathcal T(\gamma)  \}$ is a cellular basis of $A$ equivalent to the original cellular basis $\mathcal C$.  
\end{lemma}

\begin{proof}  For each $\gamma \in \Gamma$ and each $\mfs, \mft  \in \mathcal T(\gamma)$,
$b^\gamma_{\mfs, \mft}$ is a lifting in $A^\gamma$ of $\alpha_\gamma\inv(b_\mft \otimes (b_\mft )^*)$, so this follows at once from Lemma \ref{lemma: globalizing bases of cell modules}. 
\end{proof}

\subsection{Abelian cyclic cellular algebras}

\begin{lemma}  \label{lemma: properties of abelian cellular algebras}
Let $A$ be an abelian cellular algebra with cellular basis
$$
\mathcal C = \{ c_{s, t}^\gamma :  \gamma \in \Gamma \text{ and }  s, t \in \mathcal T(\gamma)\}.
$$
\begin{enumerate}
\item  For every $a \in A$ and $\gamma \in \Gamma$,  there exists $r(a) \in R$ such
that for all $x \in \Delta^\gamma$, $a x = r(a) x$.  Moreover, $a\mapsto r(a)$ is an 
$R$--algebra homomorphism from $A$ to $R$.   
\item  If $|\mathcal T(\gamma)| > 1$ 
for some $\gamma \in \Gamma$,  then
$(A^\gamma)^2 \subseteq \bar A^\gamma$.  
\end{enumerate}
\end{lemma}

\begin{proof} Let $a \in A$, $\gamma \in \Gamma$, and $s, t \in \mathcal T(\gamma)$.  
From Definition \ref{weak cellularity}
(\ref{mult rule}), we have   
$$
a c_{s, t}^\gamma  \equiv \sum_v r_v^s(a)  c_{v, t}^\gamma  \mod  \bar A^\gamma,
$$
where the coefficients $r_v^s(a)$ do not depend on $t \in \mathcal T(\gamma)$.   On the other hand, 
from Equation \ref{equation: right multiplication rule}, we have
$$a c_{s, t}^\gamma =  c_{s, t}^\gamma a 
\equiv \sum_w r_w^t(a^*)  c_{s, w}^\gamma  \mod  \bar A^\gamma,
$$   
where the coefficients $r_w^t(a^*)$ do not depend on $s \in \mathcal T(\gamma)$.  
Comparing these expansions, we have $r^s_v(a) = 0$ if $s \ne v$.  Moreover, we have
$$
a c_{s, t}^\gamma \equiv r^s_s(a) c_{s, t}^\gamma \equiv r^t_t(a^*) c_{s, t}^\gamma \mod \bar A^\gamma.
$$
It follows from this that the coefficient $r^s_s(a)$ is actually independent of $s \in \mathcal T(\gamma)$.   This implies statement (1) of the Lemma.

Suppose that  $|\mathcal T(\gamma)| > 1$ for some $\gamma \in \Gamma$.
Let 
$a, b, t \in \mathcal T(\gamma)$, with $a \ne b$.  It follows from
$c_{a, t}^\gamma c_{b, t}^\gamma = c_{b, t}^\gamma c_{a, t}^\gamma $
that
$$
\langle c_t^\gamma, c_b^\gamma\rangle c_{a, t}^\gamma \equiv
\langle c_t^\gamma, c_a^\gamma\rangle c_{b, t}^\gamma  \mod\ \bar A^\gamma.
$$
This implies  that $\langle c_t^\gamma, c_b^\gamma\rangle = 0$.  Since $b, t$ are 
arbitrary elements of $\mathcal T(\gamma)$,  this means that the bilinear form$
\langle\, , \rangle$ on $\Delta^\gamma$ is identically zero.  Statement (2) of the 
Lemma follows. 
\end{proof}

\begin{proposition} \label{proposition characterization of abelian s cellular algebras}
 Let $A$ be an abelian cellular algebra with cellular basis
$$
\mathcal C = \{ c_{s, t}^\gamma :  \gamma \in \Gamma \text{ and }  s, t \in \mathcal T(\gamma)\}.
$$  
The following are equivalent:
\begin{enumerate}
\item $A$ is cyclic cellular.
\item $|\mathcal T(\gamma)| = 1$ for all $\gamma \in \Gamma$.
\item The involution $*$ is trivial, i.e. $a^* = a$ for all $a \in A$.  
\end{enumerate}
\end{proposition}

\begin{proof}  It is obvious from properties  (1) and (3)  of Definition \ref{weak 
cellularity} that the involution on $A$ is trivial if and only if each $\mathcal 
T(\gamma)$ is a singleton.    It is also obvious that if each $\mathcal T(\gamma)$ is 
a singleton, then $A$ is cyclic cellular.    Finally, if $A$ is cyclic cellular, then for all $\gamma 
\in \Gamma$ and $s, t \in \mathcal T(\gamma)$, there exists $r \in R$ such that  
$$
  c_{s, t}^\gamma = v_s^\gamma y_\gamma (v_t^\gamma)^* \equiv r y_\gamma \mod \bar A^\gamma,
$$
by Lemma \ref{lemma: properties of abelian cellular algebras}, part (1).  Thus   
$\mathcal T(\gamma)$ is a singleton for all $\gamma \in \Gamma$. 
  \end{proof}

\begin{example}  Here is an example of an abelian cellular algebra that is not cyclic
cellular.  Let $\Gamma$ be the two element set $\{1, 2\}$,  with total order, 
$2 > 1$.  Let $\mathcal T(1) = \{1\}$ and $\mathcal T(2) = \{1, 2\}$.  Let $A$ be the 
free $\Z$--module with basis consisting of $1 = c^1_{1, 1}$  and elements 
$c^2_{s, t}$ for $1 \le s, t \le 2$.   Declare $(c^\gamma_{s, t})^* = c^\gamma_{t, s}$ 
for all $\gamma \in \Gamma$ and $s, t \in \mathcal T(\gamma)$.   Declare $1 x = x 
1 = x$ for all basis elements $x$, and all other products of basis elements to be 
zero.  Then $A$ is an abelian  cellular algebra, and $A$ is not cyclic cellular,  since $
\mathcal T(2)$ is not a singleton. 
\end{example}

\subsection{Examples of cyclic cellular algebras}  Many of the most important examples of cellular algebras are cyclic cellular.  

\begin{example} \label{example: Hecke algebra is cyclic cellular} {\em The Hecke algebra of type $A_{n-1}$.}  Fix $n \ge 1$.  Let $R$ be an integral domain and $q$ an invertible element of $R$.  The Hecke algebra $H_n = H_{n, R}(q)$  is the unital $R$--algebra with generators $T_1, \dots T_{n-1}$  satisfying the braid relations 
$T_i T_{i+1} T_i = T_{i+1} T_i T_{i+1}$ for $i \le n-2$  and $T_i T_j = T_j T_i$  when $|i-j| \ge 2$, 
and the quadratic relation $(T_i -q)(T_i +1) = 0$ for all $i$.    Let $s_i$ denote the transposition $(i, i+1)$ in 
 the symmetric group $\symn$.   Let $\pi \in \symn$ and let $\pi = s_{i_1} \cdots s_{i_\ell}$ be a reduced expression for $\pi$;  the element $T_{i_1} \cdots T_{i_\ell}$ in $H_n$ is independent of the choice of the reduced expression and is denoted by $T_\pi$.   The set of $T_\pi$ as $\pi$ varies over $\symn$ is a linear basis of $H_n$.     Because of the symmetry of the defining relations, $H_n$ has a unique $R$--linear algebra involution $*$ such that $T_i^* =T_i$ for each  $i$; we have $T_\pi^* = T_{\pi\inv}$.   
 
 Let $\Lambda_n$ denote the set of partitions of $n$;  we identify partitions with their Young diagrams.  For $\la \in \Lambda_n$,  a $\la$--tableau is a filling of the boxes or cells of $\la$ with the numbers $1, 2, \dots, n$ without repetition.   A tableau is standard if the entries increase from left to right in each row and from top to bottom in each column.  We let $\mft^\la$ denote the $\la$--tableau
 in which the numbers $1, 2, \dots, n$ are entered in order from left to right along the rows of $\la$.  
 The symmetric group $\symn$ acts transitively and freely (on the left)  on $\la$--tableaux by acting on the entries.   Therefore, for each $\la$--tableau $\mfs$ there exists a unique element $d(\mfs) \in \symn$ such that $\mfs = d(\mfs) \mft^\la$.  
 
Murphy ~\cite{murphy-hecke95} described a cellular basis of $H_n$ as follows.  The partially ordered set in the cell datum is $\Lambda_n$ with dominance order $\dominates$.   For each $\la \in \Lambda_n$, $\mathcal T(\la)$ is the set of standard $\la$--tableaux.  For $\la \in \Lambda_n$, define $m_\la = \sum_{\pi \in \sym_\la} T_\pi$,   where $\sym_\la$ is the stabilizer of $\mft^\la$.  For $\mfs, \mft \in \mathcal T(\la)$, define
$$
m^\la_{\mfs, \mft} = T_{d(\mfs)} m_\la  T_{d(\mft)}^*. 
$$
Murphy shows that $\{m^\la_{\mfs, \mft} : \la \in \Lambda_n \text{ and }  \mfs, \mft \in \Lambda_n\}$ is a cellular basis of $H_n$.   The cell module $\Delta^\la$ is spanned by $\{  T_{d(\mfs)} m_\la + \overline H_n^\la : \mfs \in \mathcal T(\la)\}$.  The cell module  $\Delta^\la$ is evidently cyclic with generator 
$m_\la + \overline H_n^\la$.  
The Hecke algebra $H_n$ with the Murphy basis is the prototypical example of a cyclic cellular algebra.
\end{example}

\begin{example} {\em  The cyclotomic Hecke algebras.}  Fix $n \ge 1$ and $r \ge 1$.  Let $R$ be an integral domain and $q$, $Q_1, \dots, Q_r$  invertible elements of $R$.  The cyclotomic Hecke algebra
$\mathcal H_n = \mathcal H_{n, R}(q, Q_1, \dots, Q_r)$  
~\cite{ariki-koike, ariki-book, DJM} is the unital $R$--algebra with generators $T_0, T_1, \dots, T_{n-1}$ such that
\begin{enumerate}
\item  $T_1, \dots, T_{n-1}$ satisfy the relations of the Hecke algebra $H_{n, R}(q)$.   
\item  $T_0 T_1 T_0 T_1 = T_1 T_0 T_1 T_0$, and $T_0$ commutes with $T_i$ for $i \ge 2$.  
\item $\prod_{j = 1}^r  (T_0 -Q_j) = 0$.  
\end{enumerate}
In ~\cite{DJM} a cellular basis of $\mathcal H_n$ is described that is analogous to the Murphy basis of the Hecke algebra $H_n$.    The cell modules defined via this basis are manifestly cyclic. 
\end{example}

\begin{example}{\em The $q$--Schur algebras.}  There are a number of different descriptions of the $q$--Schur algebras in the literature;  here we follow the approach in ~\cite{Mathas-book}, Chapter 4.\footnote{We note that our convention is that the symmetric group $\symn$ acts on the left on 
$\{1, 2, \dots, n\}$   rather than on the right as in ~\cite{Mathas-book}.  Hence our symmetric group is the opposite group of the symmetric group in ~\cite{Mathas-book}, and our Hecke algebra is the opposite algebra of the Hecke algebra in ~\cite{Mathas-book}.  Right modules in ~\cite{Mathas-book} become left modules here.}  Let $R$ be an integral domain and $q$ an invertible element in $R$.  Write $H_n$ for the Hecke algebra over $R$ with parameter $q$ as in Example \ref{example: Hecke algebra is cyclic cellular}.  For $\mu$ a composition of $n$, the permutation module $M^\mu$ of $H_n$ is  defined by $M^\mu = H_n m_\mu$.   Fix $d \le n$;  let
$C(d, n)$ denote the compositions of $n$ with no more than $d$ parts and $\Lambda(d, n)$ the partitions in $C(d, n)$;  $\Lambda(d, n)$ is partially ordered by dominance.    
The $q$--Schur algebra $\mathcal S = \mathcal S(d, n)$ is
$$
\mathcal S = \End_{H_n}\left(\bigoplus_{\mu \in C(d, n)}  M^\mu \right ).
$$
For $\la \in \Lambda(d, n)$ and $\mu \in C(d, n)$,  let $\mathcal T_0(\la, \mu)$ be the set of semistandard $\lambda$--tableaux of type $\mu$;  see ~\cite{Mathas-book}, page 56.   Let $\mathcal T_0(\la)$  denote $\cup_{\mu \in C(d, n)} \mathcal T_0(\la, \mu)$.  
One shows that for $\mu, \nu \in C(d, n)$,  $\Hom_{H_n}(M^\nu, M^\mu) \cong  (M^\nu)^* \cap M^\mu \subseteq H_n$,  as $R$--modules.  Moreover, $(M^\nu)^* \cap M^\mu$ has an $R$--basis
$\{m^\la_{\sf S, \sf T} \}$;  here $\la$ varies over all partitions  in $\Lambda(d, n)$  which admit semistandard   tableaux of types $\mu$ and $\nu$, $\sf T \in \mathcal T_0(\la, \mu)$,  $\sf S \in \mathcal T_0(\la, \nu)$,  and $m^\la_{\sf S, \sf T} $ is a certain sum of Murphy basis elements $m^\la_{\mfs, \mft}$ of $H_n$.   Corresponding to an element $m^\la_{\sf S, \sf T} \in (M^\nu)^* \cap M^\mu$, 
we obtain $\varphi^\la_{\sf S, \sf T} \in \mathcal S$ defined by $\varphi^\la_{\sf S, \sf T}(h \,m_\alpha) = \delta_{\alpha, \nu} \,h \,m^\la_{\sf S, \sf T}$.  Then
$$
\{\varphi^\la_{\sf S, \sf T}: \la \in \Lambda(d, n) \text{ and }  \sf S, \sf T \in \mathcal T_0(\la)\}
$$
is a cellular basis of $\mathcal S$. Let $\sf T^\la$ be the unique element of $\mathcal T_0(\la, \la)$;
then $m^\la_{\sf T^\la, \sf T^\la} = m_\la$ and
\begin{equation} \label{cyclic element for schur algebra cell module}
\varphi^\la_{\sf S, \sf T^\la} \, \varphi^\la_{\sf T^\la, \sf T^\la} = \varphi^\la_{\sf S, \sf T^\la},
\end{equation}
for any $\sf S \in \mathcal T_0(\la)$.   The cell module $\Delta^\la_{\mathcal S}$ of $\mathcal S$ is
the span of 
$
\{\varphi^\la_{\sf S, \sf T^\la} + \overline{\mathcal S}^\la : \sf S \in \mathcal T_0(\la)\},
$
and it follows from Equation \eqref{cyclic element for schur algebra cell module} that  the cell module is cyclic with generator $\varphi^\la_{\sf T^\la, \sf T^\la} + \overline{\mathcal S}^\la$.
\end{example}

\begin{example} {\em Algebras containing a Jones basic construction.}   In ~\cite{Goodman-Graber1}, 
a framework was developed for proving cellularity of a tower $(A_n)_{n \ge 0}$ of algebras which is obtained ``by repeated Jones basic constructions"  from a tower $(Q_n)_{n \ge 0}$ of cellular algebras. 
  In ~\cite{Enyang-Goodman},  the framework was refined by adding the assumption that the algebras $Q_n$ are cyclic cellular;  the refinement of the main theorem from   ~\cite{Goodman-Graber1} then gives that the algebras $A_n$  are cyclic cellular.  This theory applies in particular to the following examples:
Brauer algebras, BMW algebras, Jones Temperley Lieb algebras, and partition algebras. The theory   yields that all of these algebras are cyclic cellular. 

It should be noted that the cellularity these algebras was known previously, and that in general one can verify from previous proofs of cellularity that the algebras are cyclic cellular.  We refer the reader to the relevant  citations  of the literature in ~\cite{Goodman-Graber1, Enyang-Goodman}. 
\end{example}

\section{Preliminaries}
 \subsection{Combinatorial preliminaries}
 A  {\em composition} is a finite sequence of nonnegative integers, and a 
{\em partition} is a composition whose entries are non-increasing. The entries of a  
composition  are called its {\em parts}. For composition $\alpha$, we denote by 
 $|\alpha|$  the sum of the parts of $\alpha$.  Let $C_n$ denote the set of compositions $\alpha$ with $|\alpha| = n$ and $\Lambda_n$  the set of partitions $\lambda$ with $|\lambda| = n$.  We also say that
 $\alpha$ is a composition of $n$ if $\alpha \in C_n$ and that $\lambda$ is a partition of $n$ if $\lambda \in \Lambda_n$.  
 
 The {\em Young diagram} of a 
composition $\alpha = (\alpha_i)$ is the set 
 $$
  \{(i, j) : i \ge 1 \ \text{and} \ j \le \alpha_i\},
 $$
which is regarded as a left justified array of nodes or boxes in the plane, as usual.   
It is  convenient  to identify a composition with its Young diagram.
 
Let $(\Gamma,\ge)$ be a finite partially ordered set  of cardinality $r$.   Let 
$\multicomp$ denote the set of maps $\bm \lambda$ from  $\Gamma$ to the set of 
compositions  such that $\sum_{\gamma \in \Gamma}|\bm \lambda(\gamma)| = n$. Let $
\multipart \subseteq \multicomp$ denote the set of such maps such that each $
\blam(\gamma)$ is a partition.  Fix a listing of $\Gamma$,  
\begin{equation} 
\Gamma = (\gamma\paren 1,\dots, \gamma\paren r). 
\end{equation} 
We suppose (although this is not essential) that the listing is consistent with the partial 
order of  $\Gamma$, in the sense that
$$
 \gamma\paren i \ge \gamma\paren j \implies i \le j.
$$
Using  this fixed listing, we identify $\multicomp$ with sequences of compositions 
(multicompositions)  via
$$
\lambda\power i = \bm \lambda(\gamma(i)) \quad (1 \le i \le r).
$$
To each multicomposition $\blam$ we associate the composition $\alpha(\blam)$  
with parts $\alpha(\blam)_i = |\lambda\power i|$.  The Young diagram of a 
multicomposition is the sequence of Young diagrams of its components. We will identify a multicomposition with its Young diagram.

For $\blam \in \multicomp$, a  {\em $\blam$--tableau} $\mft = (\mft\power 1, \dots, \mft \power r)$is a filling of the 
nodes of $\blam$ with the numbers $\{1, \dots, n\}$ without 
repetition. The $\lambda\power k$--tableaux $\mft \power k$ are called the components 
of $\mft$.   A $\blam$--tableau is called {\em row standard} if the entries in each row of 
each component increase from left to right.  A $\blam$--tableau  $\mft$ is called {\em 
standard} if $\blam \in \multipart$, $\mft$ is row standard, and the entries in each 
column of each component  increase from top to bottom. 

Let $\blam \in \multicomp$.  The symmetric group $\symn$ acts freely and transitively on $\blam$--tableaux by acting on the entries of the tableaux.  
    Let  $\mft^{\blam}$  be the 
  $\blam$--tableau in which the entries $1, 2, \dots, n$ 
appear in order along the rows of the first component, then along the rows of the second 
component, and so forth.  The row stabilizer of $\mft^{\blam}$ (that is, the subgroup of permutations leaving invariant the set of entries in each row of  $\mft^\blam$)  is the Young subgroup $\sym_{\bm 
\lambda}$.    For each $\blam$--tableau $\mft$,  let $d(\mft)$ be the unique permutation such that
$\mft = d(\mft) \mft^{\blam}$;  then $\mft \mapsto d(\mft)$ is a bijection between $\blam$--tableaux and permutations.  The tableau $\mft$ is row standard if and only if $d(\mft)$ is a distinguished (minimal length) left coset representative of $\sym_{\blam}$ 
in $\symn$; see ~\cite{Mathas-book}, Proposition 3.3.

We define two partial orders on  $\multicomp$:
\begin{definition}  Let $\blam$ and $\bmu$ be elements of  $\multicomp$.
\begin{enumerate}
\item Dominance order $\dominates$ is  defined as in ~\cite{DJM}, Definition 3.11, namely $\blam \dominates \bmu$ if for all $k$  ($1 \le k \le r$)  and all $j \ge 0$,
$$
\sum_{s < k} |\lambda\power s| + \sum_{i \le j} \lambda\power k_i \ge
\sum_{s < k} |\mu\power s| + \sum_{i \le j} \mu\power k_i.
$$
We write $\blam \StrictlyDominates \bmu$ if $\blam \dominates \bmu$ and $\blam \ne \bmu$.  
\item $\Gamma$--dominance order $\GammaDominates$ is  defined similarly:  
$\blam \GammaDominates \bmu$ if for all $\gamma \in \Gamma$,  and for all $j \ge  0$, 
$$
\sum_{\gamma' > \gamma} |\blam(\gamma')| + \sum_{i \le j} \blam(\gamma)_i \ge
\sum_{\gamma' > \gamma} |\bmu(\gamma')| + \sum_{i \le j} \bmu(\gamma)_i .
$$
We write $\blam \StrictlyGammaDominates \bmu$ if $\blam \GammaDominates \bmu$ and $\blam \ne \bmu$.  
\end{enumerate}
\end{definition}

\begin{example}  
Neither of the conditions $\blam \dominates \bmu$ and 
$\blam \GammaDominates \bmu$ implies the other.
\begin{enumerate}
\item  Let $\Gamma = \{a, b\}$, with $a$ and $b$ incomparable. Choose the listing 
$\Gamma = (a, b)$.  Let $\blam(a) = (2)$, 
$\blam(b) = \emptyset$ and $\bmu(a) = \bmu(b) = (1)$.  Then 
$\blam \dominates \bmu$   but 
$\blam$ and $\bmu$ are incomparable in $\Gamma$-dominance order.
\item Let $\Gamma = \{a, b, c, d\}$ with $a > x$ for each $x \in \{b, c, d\}$ and 
$\{b, c, d\}$  mutually 
incomparable.  Choose the listing $\Gamma = (a, b, c, d)$.   Let $\blam(a) = (4)$, $\blam(b) = (2)$, $
\blam(c) = (1)$, $\blam(d) = (1)$ and let $\bmu(a ) = (3)$, $\bmu(b ) = (3)$, $\bmu(c ) = (2)$, $
\bmu(d ) = (0)$.  Then $\blam \GammaDominates \bmu$ but $\blam$ and $\bmu$ are incomparable 
in dominance order.  In this example, $\blam$ cannot be obtained from $\bmu$ by any sequence of 
``raising operators,"  c.f Macdonald ~\cite{macdonald-book}, Section I.1.
\end{enumerate}  
\end{example}

We extend these partial orders to partial orders on the set of row standard tableaux.
 If $\mft$ is a row standard $
\blam$--tableau, then for each $m \le n$,  the entries $1, 2, \dots, m$ in $\mft$   occupy 
the diagram of a multicomposition, say $\bmu$;  the $\bmu$--tableau obtained by removing the nodes containing $m+1, \dots, n$ from $\mft$ is denoted by $\mft \downarrow m$;  we also denote 
$\bmu$ by 
$[\mft  \downarrow m]$.

\begin{definition}  Let $\blam$ and $\bmu$ be elements of  $\multicomp$, and let 
 $\mft$ be a   row   standard  $\blam$--tableau, and $\mfs$ a row standard 
 $\bmu$--tableau. 
\begin{enumerate}
\item We say that {\em $\mft$  dominates $\mfs$} and write
$\mft \dominates \mfs$, if $[\mft \downarrow m] \dominates [\mfs \downarrow m]$
 for each $m$.  
We write $\mft \StrictlyDominates \mfs$ if $\mft \dominates \mfs$ and 
$\mft \ne \mfs$.   

\item Let $\mft$ be a   row   standard  $\blam$--tableau, and $\mfs$ a row standard 
$\bmu$--tableau.  We say that {\em $\mft$  $\Gamma$--dominates $\mfs$} 
and write
$\mft \GammaDominates \mfs$, if 
$[\mft \downarrow m]\GammaDominates [\mfs \downarrow m]$
 for each $m$. We write $\mft \StrictlyGammaDominates \mfs$ if 
 $\mft \GammaDominates \mfs$ 
 and $\mft \ne \mfs$.   
\end{enumerate}
\end{definition}

\begin{definition}  Let $\blam \in \multicomp$ and let $\mft$ be a $\blam$--tableau. 
 Define the 
component function of $\mft$ by $\comp_{\mft}(j) = k$ if $j$ appears in the $k$--th 
component of $\mft$.  For a $\blam$--tableau $\mft$ and a $\bmu$--tableau 
$\mfs$,  write $\comp_\mft = \comp_\mfs$ if $\comp_\mft(j) = \comp_\mfs(j)$ 
for all $j$.  
\end{definition}

\begin{definition}  \label{definition initial kind}
A $\blam$--tableau $\mft$ is {\em of the initial kind} if 
$\comp_\mft = \comp_{\mft^{\blam}}$.  Equivalently,  $d(\mft) \in \sym_{\alpha(\blam)}$.  
\end{definition}

\begin{remark}   \label{remark on dominance and gamma dominance}
 If $\blam$ and $\bmu$ satisfy $|\blam(\gamma)| = |\bmu(\gamma)|$ for all 
 $\gamma$ (i.e. $\alpha(\blam) = \alpha(\bmu)$),  then
$$
\blam \dominates \bmu \iff \blam \GammaDominates \bmu \iff \blam(\gamma) \dominates \bmu(\gamma) \ \ \text{for all} \ \ \gamma.
$$
Suppose that 
$\mft$ is  a $\blam$--tableau and  $\mfs$  a $\bmu$--tableau,  with 
$\comp_\mft = \comp_\mfs$.  
It follows that $\alpha([\mft\downarrow m]) = \alpha([\mfs \downarrow m])$ for all $m$.
 Therefore
$$
\mft \dominates \mfs \iff \mft \GammaDominates \mfs \iff \mft\power k \dominates \mfs\power k  \ \ \text{for all} \ \ k.
$$
\end{remark}

\begin{remark}   All of the definitions of this section apply in particular when $r = 1$.  Thus we have defined $\lambda$--tableaux and row standard $\lambda$--tableaux for $\lambda \in C_n$, standard $\lambda$--tableaux for $\lambda \in \Lambda_n$,  dominance order on compositions and on row standard tableaux, and the ``superstandard" tableau $\mft^\lambda$ for $\lambda \in C_n$.  
\end{remark}

\subsection{Cellularity  of tensor products}  \label{subsection: cellularity of tensor products}
Let $A\power i$ ($1 \le i \le K$) be cellular algebras over an integral domain $R$, with  algebra involutions denoted by $*$,  partially ordered sets $(\Gamma_i, \ge)$,  index sets $\mathcal T(\gamma)$  for $\gamma \in \Gamma_i$,   and cellular bases $\{a^\gamma_{s, t} :  \gamma \in \Gamma_i,  s, t \in \mathcal T(\gamma)\}$.  Then 
 $A\power 1 \otimes \cdots \otimes A \power K$ is also a cellular algebra.  The involution on the tensor product algebra is the tensor product of the involutions on each $A \power i$.  The partially ordered set in the cellular structure for the tensor product algebra is the cartesian product
 $\Pi = \Gamma_1 \times \cdots \times \Gamma_K$,  with the product partial order, namely
 $(\gamma_1, \dots, \gamma_K) \ge (\gamma'_1, \dots, \gamma'_K)$  if $\gamma_i \ge \gamma'_k$ for all $i$.   For each $\bgamma = (\gamma_1, \dots, \gamma_K) \in \Pi$,  the index set
 $\mathcal T(\bgamma)$ is $\mathcal T(\gamma_1) \times \cdots \times \mathcal T(\gamma_K)$. 
 The cellular basis is the set of simple tensors
 $$
 a^{\bgamma}_{s, t} = a^{\gamma_1}_{s_1, t_1} \otimes \cdots \otimes a^{\gamma_K}_{s_K, t_K},
 $$
 for $\bgamma = (\gamma_1, \dots, \gamma_K) \in \Pi$, and 
 $s = (s_1, \dots, s_K), t = (t_1, \dots, t_K) \in \mathcal T(\bgamma)$.   The cellular structure of the tensor product algebra is strict if each $A\power i$ is strictly cellular.  If $\Delta^\gamma$  ($\gamma \in \Gamma_i$)  denote the cell modules of each $A\power i$,  then the cell modules of the tensor product algebra are
 $$
 \Delta^{\bgamma} =  \Delta^{\gamma_1} \otimes \cdots \otimes  \Delta^{\gamma_K},
 $$
 for $\bgamma = (\gamma_1, \dots, \gamma_K)  \in \Pi$.  
 
 Suppose that each $A\power i$ is cyclic cellular with data
 $$
 \generator {} \gamma, \  V^\gamma = \{v_s : s \in \mathcal T(\gamma)\} \text{ and }  y_\gamma,
 $$
 for each $\gamma \in \Gamma_i$, as in Notation \ref{notation cyclic cellular}.  
 For  $\bgamma = (\gamma_1, \dots, \gamma_K) \in \Pi$ and
 $\bm s = (s_1, \dots, s_K) \in \mathcal T(\bgamma)$, define
 $$
 v_{\bm s} = v_{s_1} \otimes \cdots \otimes v_{s_K}.
 $$
 Then the tensor product algebra is cyclic cellular as well, with
 $$
 \generator {} \bgamma = \generator {} {\gamma_1} \otimes \cdots \otimes \generator {} {\gamma_K}, \ 
 V^{\bgamma} = \{v_{\bm s} : \bm s \in \mathcal T(\bgamma)\} , \text{ and }
 y_{\bgamma} = y_{\gamma_1} \otimes \cdots  \otimes y_{\gamma_K}.
 $$

\subsection{
The Murphy basis of $R\symn$}  \label{subsection: Murphy basis}
Let $R$ be an integral domain.  Fix $n \ge 1$.  Consider the set $\Lambda_n$ of partitions of $n$ with dominance order.  For each $\lambda \in \Lambda_n$,  let $\mathcal T(\lambda)$ be the set of standard $\lambda$--tableaux.  Let $*$ be the algebra involution of $R \symn$ such that $\pi^* = \pi\inv$ for $\pi \in \symn$.  For $\lambda \in \Lambda_n$,  let 
$x_\lambda = \sum_{w \in \sym_\lambda} w$,  where $\sym_\lambda$ is the row stabilizer of $\mft^\lambda$.   For $\mfs, \mft \in \mathcal T(\lambda)$, define
$$
m^\lambda_{\mfs, \mft} = d(\mfs) x_\lambda d(\mft)^*.
$$
The symmetric group algebra is the special case of the Hecke algebra $H_{n, R}(q)$ with $q = 1$, and the collection of elements $m^\lambda_{\mfs, \mft}$ is the specialization of the Murphy basis of the Hecke algebra, see Example \ref{example: Hecke algebra is cyclic cellular}.  Thus
 $\{m^\lambda_{\mfs, \mft} : \lambda \in \Lambda_n, \mfs, \mft \in \mathcal T(\lambda)\}$ is a cellular basis of $R\symn$, and $R\symn$ is cyclic cellular.

\subsection{The cellular structure of $R\sym_\alpha$} 
\label{subsection the cellular structure of the group algebra of a Young subgroup}
 Let $R$ be an integral domain.  Let $\alpha = (\alpha_1, \dots, \alpha_r)$ be a composition of an integer $n$, with $r$ parts.  Let $\sym_\alpha \subseteq \symn$ be the Young subgroup, the row stabilizer of $\mft^\alpha$.  
Let $n_k = \alpha_1 + \cdots + \alpha_{k-1}$  for $1 \le k \le r$ and for each $k$ realize $\sym_{\alpha_k}$ as the symmetric group of the set $\{n_k + 1, \dots, n_{k+1} \}$.  
Then $ \sym_{\alpha_1} \times \cdots \times \sym_{\alpha_r} \cong \sym_\alpha$ by
$(\pi_1, \dots, \pi_r) \mapsto \pi_1 \pi_2 \cdots \pi_r$, and 
$R\sym_{\alpha_1} \otimes \cdots \otimes R\sym_{\alpha_r} \cong R\sym_\alpha$ by
\begin{equation}\label{R S alpha correspondence}
\pi_1 \otimes \cdots \otimes \pi_r \mapsto  \pi_1 \pi_2 \cdots \pi_r.
\end{equation}  
By the discussion in Sections \ref{subsection: cellularity of tensor products} and 
\ref{subsection: Murphy basis}, the tensor product algebra 
$R\sym_{\alpha_1} \otimes \cdots \otimes R\sym_{\alpha_r} $ is cellular;  the partially ordered set for the cellular structure is $\Lambda_{\alpha_1} \times \cdots \times \Lambda_{\alpha_r}$ with the product of dominance partial order on each component;  but this is the same as the set 
$\Lambda_\alpha$ of multipartitions $\blam$ with $\alpha(\blam) = \alpha$,  with dominance order of multipartitions, by Remark \ref{remark on dominance and gamma dominance}.  For $\blam  = (\la\power 1, \dots, \la \power r) \in \Lambda_\alpha$,  the index set 
$\mathcal T(\blam)$ is $\mathcal T(\la\power 1) \times \cdots \times \mathcal T(\la\power r)$;  but this can be identified with the set of standard $\blam$--tableaux of the initial kind.

The cellular basis is the tensor product of the Murphy cellular basis of each $R\sym_{\alpha_i}$.  Under the isomorphism \refc{R S alpha correspondence},  this basis is carried to the set of elements
$$
m^\blam_{\mfs, \mft} = d(\mfs) x_\blam d(\mft)^*,
$$
where $\blam \in \Lambda_\alpha$, and $\mfs$ and $\mft$ are standard $\blam$--tableaux of the initial kind.
The cell modules for the tensor product algebra $R\sym_{\alpha_1} \otimes \cdots \otimes R\sym_{\alpha_r} $ are the modules
$$
\Delta^{\blam}_R =  \Delta^{\lambda\power 1}_r \otimes \cdots \otimes  \Delta^{\lambda\power r}_R,
$$
where $\blam \in \Lambda_\alpha$ and 
$\Delta^{\lambda\power i}_R$ denotes a cell module for $R\sym_{\alpha_i}$.   

We are going to write $J^\blam$ for $\overline{R\sym_\alpha}^\blam = \spn\{\mm \bmu \mfu \mfv :
\bmu \StrictlyDominates \blam \text{ and } \mfu, \mfv \in \mathcal T(\bmu)\}$. 
Under the identification of the tensor product algebra with $R\sym_\alpha$,  the cell module $\Delta^{\blam}_R$ can be identified with 
$$
\spn\{ \mm \blam \mfs {\mft^\blam}  +  J^\blam\} ,
$$
where $\mfs$ varies over standard $\blam$--tableaux of the initial kind.

 \section{Wreath products}\label{Cellularity of wreath products}  
 Let $A$ be an $R$--algebra.  The symmetric group $\symn$ acts by algebra automorphisms on the tensor product algebra $A^{\otimes n}$, by place permutations,
 $$
 \leftexp{\pi}{\hskip - .1 em (a_1\otimes \cdots \otimes a_n)} = a_{\pi\inv(1)} \otimes \cdots \otimes a_{\pi\inv(n)}.
 $$
 The wreath product $A \wr \symn$ is the crossed product algebra $A^{\otimes n} \rtimes \symn$.  If $A$ is an algebra with involution $*$,  then $\symn$ acts by $*$--preserving automorphisms and the wreath product is also an algebra with involution determined by
 $$
 \begin{aligned}
 ((a_1\otimes \cdots \otimes a_n) \pi)^* &= \pi\inv \,(a_1^* \otimes \cdots \otimes a_n^*) \\
 &= \leftexp{\pi\inv}{\hskip - .2 em (a_1^* \otimes \cdots \otimes a_n^*) } \, \pi\inv.
 \end{aligned}
 $$
 An $R$--module $M$ is a module for the wreath product $A\wr\symn$ if and only if $M$ is both an
 $A^{\otimes n}$--module and a $\symn$--module, and the following covariance condition is satisfied:
 $$
 \pi (a m) = \leftexp{\pi} {\hskip -.1 em a} (\pi m),
 $$
 for $\pi \in \symn$,  $a \in A^{\otimes n}$,  and $m \in M$.  
 
 \subsection{Cellularity of wreath products} \label{subsection: cellularity theorem for wreath products}
 In this section, we will prove the following theorem:

\begin{theorem}  \label{theorem cellularity of wreath products}
Let $A$ be a cyclic cellular  algebra.  Then for all $n \ge 1$,  the wreath product algebra
$A \wr \sym_n$ is a cyclic cellular  algebra.  
\end{theorem}

In  Dipper, James, and Mathas ~\cite{DJM}, Section 3, it is shown that the cyclotomic 
Hecke algebras are strictly cellular.  This theorem contains as a special case the 
statement that the wreath product algebra $RC_r \wr \symn$ is strictly cellular, where 
$C_r$ denotes the cyclic group with $r$ elements, and $R$ is assumed to contain a 
primitive $r$--th root of unity.  In proving Theorem \ref{theorem cellularity of wreath 
products}, we adapt the proof of ~\cite{DJM}, Theorem 3.26.

 Let $A$ be a  cyclic cellular algebra, over an integral domain  $R$,  with cell datum
\break $(A, *, \Gamma, \ge, \mathcal T, \mathcal C)$.   For $\gamma \in \Gamma$, let
 $$
 \generator {} \gamma, \ 
V^\gamma  = \{v_j: j \in  \mathcal T(\gamma)\}, \ \  \text{and} \ \ y_\gamma 
 $$
 be as described in Notation \ref{notation cyclic cellular}.
  
 First we are going to define some elements of $A \wr  \sym_n$ which are analogues of the Murphy 
elements $m^\lambda_{\mfs, \mft}$ (in the Hecke algebra or in the symmetric group algebra) associated to a 
composition $\lambda$ and a pair of row standard $\lambda$--tableaux, see ~\cite{murphy-hecke95}, 
Section 4. 
 
  Let $r$  denote the cardinality of $\Gamma$, and let
 $ \Gamma = (\gamma\paren 1,\dots, \gamma\paren r)$ be a listing of $\Gamma$ such that
 $\gamma(i) \ge \gamma(j)$ implies $i \le j$.

   Given $\blam \in \multicomp$, let 
  $\alpha = \alpha(\blam)$.  
For $1 \le k \le r$, let $n_k = \alpha_1 + \dots + \alpha_{k-1}$.    
Let $V^\alpha$ be the 
set of simple tensors $w = w_1\otimes\cdots \otimes w_n$ in $A^{\otimes n}$ 
such that $w_i \in 
V^{\gamma\paren k}$ if $n_k + 1 \le i \le n_{k+1}$; that is, the first $\alpha_1$ 
tensorands belong to 
$V^{\gamma\paren 1}$,  the next $\alpha_2$ tensorands belong to  $V^{\gamma\paren 2}$, 
and so forth.  Note that 
  $\leftexp{\pi}{\hskip - .1 em (V^\alpha)} = V^\alpha$ for all $\pi \in \sym_\alpha$.   Set
$$
y^\alpha = y_{\gamma\paren 1}^{\otimes \alpha_1} \otimes \cdots  \otimes y_{\gamma\paren r}^{\otimes \alpha_r}. 
$$
Observe that $\leftexp{\pi}{\hskip -.1 em (y^\alpha)} = y^\alpha$ for $\pi \in \sym_\alpha$. 

\begin{remark}\label{remark on V alpha}
If we let $\bgamma(\alpha) \in \Gamma^n$ be the sequence whose first $\alpha_1$ entries are equal to $\gamma\paren 1$, whose next $\alpha_2$ entries are equal to $\gamma\paren 2$, and so forth, then in the notation of Section \ref{subsection: cellularity of tensor products},  $V^\alpha = V^{\bgamma(\alpha)}$ and $y^\alpha = y_{\bgamma(\alpha)}$.  
\end{remark}

Define
 $$
x_{\blam} = \sum_{w \in \sym_{\blam }} w \in R \symn.
$$
Note that $x_\blam^* = x_\blam$  and $\pi x_\blam = x_\blam \pi = x_\blam$ for 
$\pi \in \sym_\blam$. 
Since elements of $\sym_\alpha$ commute with $y^\alpha$,  and 
$\sym_\blam \subseteq \sym_\alpha$ 
we have $x_\blam y^\alpha = y^\alpha x_\blam$.

We consider the following elements of $A \wr \symn$,
$$
\m \blam  \mfs v \mft w =  d(\mfs) v \,y^\alpha x_{\blam} \, w^* d(\mft)^*,
$$
where $\blam \in \multicomp$, $\alpha = \alpha(\blam)$, $\mfs, \mft$ are row standard $\blam$--tableaux,
 and $v, w \in V^\alpha$.   We will also need to consider the elements
 $$
\m \blam  \mfs v \mft 1 =  d(\mfs) v \,y^\alpha x_{\blam} \,  d(\mft)^*,
$$
where $w \in V^\alpha$  has been replaced by the multiplicative identity.  In order to facilitate statements that apply to both types of elements, we write $V^\alpha_1 = V^\alpha \cup \{1\}$.

\ignore{
\begin{lemma}  \label{lemma involution on general murphy type elements}
Let $\blam \in \multicomp$, and $\alpha = \alpha(\blam)$.  For  $\mfs, \mft$  row standard $\blam$--tableaux,
 and $v, w \in V^\alpha$, 
$$(\m \blam \mfs v \mft w)^* = \m \blam \mft w \mfs v.$$
\end{lemma}
}

\ignore{
\begin{proof}  Since $y_\gamma^* = y_\gamma$ for all $\gamma \in \Gamma$, we also have 
$(y^\alpha)^* = y^\alpha$.  Moreover, $(x_\blam)^* = x_\blam$, and
$x_\blam$ and $y^\alpha$ commute.  The conclusion follows from these observations.
\end{proof}
}

\begin{lemma}  \label{lemma on action of symn on murphy type elements 1}
Let $\blam \in \multicomp$ and  $\alpha = \alpha(\blam)$.  Let $\mfs, \mft$ be row standard 
$\blam$--tableaux, and let $v \in V^\alpha$ and $w \in V^\alpha_1$.    For $\pi \in \sym_n$,
$$
\pi \, \m \blam \mfs v \mft w =  \m \blam {\mathfrak s'} {v'} \mft w,
$$
for some row standard $\blam$--tableau $\mathfrak s'$ and some $v' \in V^\alpha$.  
Similarly,  for $\mfs, \mft$ as above and $v, w \in V^\alpha$, 
$$
\m \blam \mfs v \mft w  \pi =  \m \blam {\mathfrak s} {v} {\mft'} {w'},
$$
for some row standard $\blam$--tableau $\mathfrak t'$ and some $w' \in V^\alpha$.
\end{lemma}

\begin{proof}  $\pi d(\mfs) = d(\mathfrak s') \pi'$,  where $\pi' \in \sym_{\blam}$ and
$d(\mathfrak s')$ is a distinguished left coset representative of $ \sym_{\blam}$ in $\sym_n$.  Also, 
$\pi' v y^\alpha x_\blam = \leftexp{\pi'}{\hskip -.3 em v} 
\,\leftexp{\pi'}{\hskip - .3 em y}{^\alpha} \,\pi' x_\blam = 
\leftexp{\pi'}{\hskip -.3 em v} \, y^\alpha \, x_\blam$, 
since $y^\alpha$ is fixed by 
$\pi' \in \sym_\blam$ and $\pi' x_\blam = x_\blam$.  

Thus
$$
\begin{aligned}
\pi \m \blam \mfs v \mft w &=  \pi d(\mfs) v y^\alpha x_\blam w^*  d(\mft)^*\\
&=  d(\mathfrak s') \pi' v y^\alpha x_\blam w^*  d(\mft)^*\\
&= d(\mathfrak s') v' y^\alpha x_\blam w^*  d(\mft)^* =  \m \blam {\mathfrak s'} {v'} \mft w,\\
\end{aligned}
$$
where $v' =  \leftexp{\pi'}{\hskip -.3 em v} $.  
The statement concerning multiplication by $\pi$ on the right follows by a similar computation.
\end{proof}

\begin{notation}  For $\blam \in \multipart$,  let $\mathcal T(\blam)$  denote the set of pairs
$(\mfs, v)$,  where $\mfs$ is a standard $\blam$--tableau and $v \in V^{\alpha(\blam)}$.  
\end{notation}

The statement and proof of the following result are adapted from ~\cite{DJM}, Proposition 3.18, which in turn
 generalizes a theorem of Murphy, ~\cite{murphy-hecke95}, Theorem 4.18.

\begin{proposition} \label{lemma on expansion of general murphy elements in standard elements}
 Let $\blam \in \multicomp$, and let $\alpha = \alpha(\blam)$.  Suppose that  $\mfs, \mft$ are row standard 
 $\blam$--tableaux and $v, w \in V^{\alpha}$.   Then $\m \blam \mfs v \mft w$ is a linear combination of 
 elements $\m \bmu \mfu {v'} \mfv {w'}$, where: 
\begin{enumerate}
\item
$\bmu \in \multipart$ with $\alpha(\bmu) = \alpha$, and $\bmu \GammaDominates \blam$. 
\item $(\mfu, v')$ and $(\mfv, w')$ are elements of $\mathcal T(\bmu)$.
\item $\mfu \GammaDominates \mfs$,  and  $\comp_\mfu = \comp_\mfs$. 
\item  $\mfv \GammaDominates \mft$ and $\comp_\mfv = \comp_\mft$.
\end{enumerate}
\end{proposition}

\begin{proof} Consider first the special case that $\mfs$  and $\mft$ are of the initial kind for $\blam$ 
(see Definition \ref{definition initial kind}).    Thus $d(\mfs), d(\mft) \in \sym_\alpha$,  and
\begin{equation} \label{murphy type lemma eqn 1}
\begin{aligned}
\m \blam \mfs v \mft w &= d(\mfs) v y^\alpha x_\blam w^*  d(\mft)^* \\
&=  (\leftexp {d(\mfs) }{\hskip - .1 em v}) \, y^\alpha  d(\mfs) x_\blam d(\mft)^* \,  (\leftexp {d(\mft) }{\hskip - .1 em w}^*).\\ 
\end{aligned}
\end{equation} 
Recall the notation  $n_1 = 0$, and   $n_k = \alpha_1 + \dots + \alpha_{k-1}$ for $k \ge 2$.  We can write 
$d(\mfs) x_\blam {d(\mft)}^*$ as the product of $r$ commuting factors, say  
$d(\mfs) x_\blam {d(\mft)}^* = x_1 \cdots x_r$, where  for each $k$,  $x_k$ is contained in the group algebra
of the symmetric group on the letters $\{n_k +1, \dots, n_{k+1}\}$.  Indeed, 
$x_k =  d(\mfs\power k)  x_{\lambda\power k} {d(\mft \power k)}^*$,  where $\lambda\power k$ is the $k$--th 
component of $\blam$,  $\mfs\power k$ the $k$--th component of $\mfs$, and $\mft \power k$ the $k$--th 
component of $\mft$.  
  
Applying the theorem of Murphy,  ~\cite{murphy-hecke95}, Theorem 4.18, to the symmetric group algebra on 
the letters  $\{n_k +1, \dots, n_{k+1}\}$,  we can write each $x_k$ as a linear combination of elements $x_k' 
= d(\mfu\power k) x_{\mu\power k} d(\mfv\power k)^*$,  where
$\mu\power k$ is a partition of size $\alpha_k$,  $\mfu\power k$ and $\mfv \power k$ are
standard $\mu\power k$--tableaux  (but  with entries in $\{n_k +1, \dots, n_{k+1}\}$), and
$\mfu \power k\dominates \mfs \power k$ and $\mfv  \power k \dominates \mft \power k$.

For any choice of $x_k' = d(\mfu\power k) x_{\mu\power k} {d(\mfv\power k)}^*$, for $1 \le k \le r$,  let 
$x' = x_1'x_2'\cdots x_r'$ and  let  $\bmu$ be the multipartition  whose $k$--th component is $\mu\power k$ 
for each k,  $\mfu$ the $\bmu$--tableau whose $k$--th component is $\mfu\power k$ for each $k$, and 
$\mfv$ the $\bmu$--tableau whose $k$--th component is $\mfv\power k$ for each $k$.
Then $x' = d(\mfu) x_\bmu {d(\mfv)}^*$. Since $\mfu$ and $\mfv$ are of the initial kind  (for $\blam$) and 
$\mfu \power k\dominates \mfs \power k$ and $\mfv  \power k \dominates \mft \power k$, we have $\mfu 
\dominates \mfs$ and $\mfv \dominates \mft$,  by Remark \ref{remark on dominance and gamma 
dominance}. (Equivalently, by the same remark, $\mfu \GammaDominates \mfs$ and $\mfv 
\GammaDominates \mft$.) It follows that $d(\mfs) x_\blam {d(\mft)}^*$ is a linear combination of terms 
$d(\mfu) x_\bmu {d(\mfv)}^*$,  where $\bmu$,  $\mfu$ and $\mfv$ are as in the statement of the lemma.  
Thus from Equation \ref{murphy  type lemma eqn 1}, we have that $\m \blam \mfs v \mft w$ is a linear 
combination of terms
\begin{equation} \label{murphy type lemma eqn 2}
\begin{aligned}
(\leftexp {d(\mfs) }{\hskip - .1 em v}) \, y^\alpha  d(\mfu) x_\bmu d(\mfv)^* \,  (\leftexp {d(\mft) }{\hskip - .1 em w}^*)  &= 
d(\mfu)\, (\leftexp {d(\mfu)^* d(\mfs) }{\hskip - .1 em v}) \, y^\alpha  x_\bmu   (\leftexp {d(\mfv)^*d(\mft) }{\hskip - .1 em w}^*) d(\mfv)^* \\
&= d(\mfu)\, v' \, y^\alpha  x_\bmu  (w')^* d(\mfv)^* \\
&= \m \bmu \mfu {v'} \mfv {w'}.
\end{aligned}
\end{equation}
This completes the proof in the special case that $\mfs$ and $\mft$ are of the initial kind for $\blam$.  

We pass to the general case, where the component functions $\comp_\mfs$ and $\comp_\mft$ are arbitrary.
We can write $\mfs = \sigma \mfs'$  and $\mft = \tau \mft'$, where $\mfs'$ and $\mft'$ are row standard 
$\blam$--tableaux of the initial kind, and $\sigma, \tau$ are
distinguished left coset representatives of $\sym_\alpha$ in $\symn$.  Furthermore, 
$d(\mfs) = \sigma d(\mfs')$ and $d(\mft) = \tau d(\mft')$.  It follows that
$$
\m \blam \mfs v \mft w  = \sigma \m \blam {\mfs'} v {\mft'} w  \tau^*.
$$

Applying the result for the special case above to $\m \blam {\mfs'} v {\mft'} w$, we have that
$\m \blam \mfs v \mft w$ is a linear combination of terms
\begin{equation}  \label{murphy type lemma eqn 3}
\sigma \m \bmu {\mfu'} {v'} {\mfv'} {w'} \tau^* = 
\sigma d(\mfu') v' y^\alpha x_\bmu (w')^* d(\mfv') \tau^*, 
\end{equation}
where:  $\bmu \in \multipart$,  $\alpha(\bmu) = \alpha$, and $\bmu \dominates \blam$;   $\mfu', \mfv'$  are 
standard $\bmu$--tableaux of the initial kind such that $\mfu' \dominates \mfs'$ and
$\mfv' \dominates \mft'$; and $v', w' \in V^\alpha$.  For each such term, let $\mfu = \sigma \mfu'$ and $\mfv 
= \tau \mfv'$.  Then $\mfu$ and $\mfv$ are standard $\bmu$--tableaux with $\comp_\mfu = \comp_\mfs$ and
$\comp_\mfv = \comp_\mft$.  We have $\sigma d(\mfu') = d(\mfu)$ and $\tau d(\mfv') = d(\mfv)$, so the term  
$\sigma \m \bmu {\mfu'} {v'} {\mfv'} {w'} \tau^*$   from Equation \ref{murphy type lemma eqn 3} is equal to 
$$
 d(\mfu) v' y^\alpha x_\bmu (w')^* d(\mfv) = \m \bmu {\mfu} {v'} {\mfv} {w'}.
$$
By ~\cite{DJM}, Lemma 3.17, part (iii), we have
$\mfu \dominates \mfs$ and $\mfv \dominates \mft$.   But since 
$\comp_\mfu = \comp_\mfs$ and
$\comp_\mfv = \comp_\mft$, this is equivalent to 
$\mfu \GammaDominates \mfs$ and
$\mfv \GammaDominates \mft$, using  Remark \ref{remark on dominance and gamma dominance} again.

Thus $\m \blam \mfs v \mft w$ is a linear combination of terms of the required type.
\end{proof}

\begin{corollary} \label{lemma  2 on expansion of general murphy elements in standard elements}
Let $\blam \in \multicomp$,  , and write $\alpha = \alpha(\blam)$.  Suppose that  $\mfs, \mft$ are row 
standard $\blam$--tableaux and $v \in V^{\alpha}$.   Then
 $\m \blam \mfs v \mft 1$ is a linear combination of elements $\m \bmu \mfu {v'} \mfv {1}$, where 
$\bmu$, $\mfu, \mfv$, and $v'$ are as in the statement of  Proposition 
\ref{lemma on expansion of general murphy elements in standard elements}.
\end{corollary}

\begin{proof}  Trace through the proof of Proposition \ref{lemma on expansion of general murphy elements in standard elements} with $w \in V^\alpha$ replaced by the multiplicative identity.
\end{proof}

\begin{corollary} \label{corollary on action of symn on murphy type elements 2} 
 Let $\blam \in \multicomp$, and write $\alpha = \alpha(\blam)$. Suppose that $\mfs, \mft$ are row standard
$\blam$--tableaux, and $v, w \in V^\alpha$.    Let $\pi \in \symn$.
\begin{enumerate}
\item
$\pi \m \blam \mfs v \mft w$ is a linear combination of terms 
$\m \bmu \mfu {v'} \mfv {w'}$, where: 
\begin{enumerate}
\item
$\bmu \in \multipart$,  $\alpha(\bmu) = \alpha$, and $\bmu \GammaDominates \blam$.
\item $(\mfu, v')$ and $(\mfv, w')$ are elements of $\mathcal T(\bmu)$.
\item $\mfv \GammaDominates \mft$ and $\comp_\mfv = \comp_\mft$.
\end{enumerate} 
\item $ \m \blam \mfs v \mft w\, \pi$ is a linear combination of terms 
$\m \bmu \mfu {v'} \mfv {w'}$, where  $\bmu$, $(\mfu, v')$,  and $(\mfv, w')$ satisfy conditions  {\rm(a)} and {\rm(b)} of part  {\rm(1)}  as well as:
\begin{itemize}
\renewcommand{\labelitemi}{(c$'$)}
\item $\mfu \GammaDominates \mfs$ and $\comp_\mfu = \comp_\mfs$.
\end{itemize}

\item  For $\pi \in \symn$, 
$\pi \m \blam \mfs v \mft 1$ is a linear combination of terms 
$\m \bmu \mfu {v'} \mfv {1}$, where
$\bmu$, $\mfu, \mfv$, and $v'$ are as in part {\rm(1)}.
\end{enumerate} 
\end{corollary}

\begin{proof}  Parts (1) and  (2) result from combining Lemma \ref{lemma on action of symn on 
murphy type elements 1} and Proposition \ref{lemma on expansion of general murphy 
elements in standard elements}.    Part (2) uses Corollary \ref{lemma  2 on expansion of 
general murphy elements in standard elements} as well. 
\end{proof}

We define certain $R$--submodules of the wreath product algebra $A \wr \symn$.  We will 
show  later that these are two sided ideals of $A \wr \symn$.

\vbox{
\begin{definition} \label{definition of standard ideals in wreath product algebra}
 Let $\blam \in \multicomp$.
\begin{enumerate}
\item  Let $N^\blam$ be the $R$--submodule of $A \wr \symn$ spanned by the set of 
$\m \bmu \mfs v \mft w$  where $\bmu \in \multipart$,  $\bmu \GammaDominates \blam$,
and $(\mfs,v),  (\mft, w) \in \mathcal T(\bmu)$.
\item Let $\overline{N}^\blam$ be the $R$--submodule of $A \wr \symn$ spanned by the set of 
$\m \bmu \mfs v \mft w$  where $\bmu \in \multipart$,  $\bmu \StrictlyGammaDominates \blam$,
and $(\mfs,v),  (\mft, w) \in \mathcal T(\bmu)$.
\end{enumerate}
\end{definition}
}

\begin{corollary}  \label{final lemma on action of symn on murphy type elements} 
Let $\blam \in \multipart$ and write $\alpha = \alpha(\blam)$.  Let $\mfs$ be a standard $\blam$--tableau and $v \in V^\alpha$.  Then for $\pi \in \symn$, 
$$
\pi \m  \blam \mfs v {\mft^\blam} 1 = x_1 + x_2, 
$$
where
$x_1$ is a linear combination of elements $\m \blam  \mfu {v'} {\mft^\blam} 1$, with
$(\mfu, v') \in \mathcal T(\blam)$, and $x_2 \in \overline N^\blam$.  
\end{corollary}

\begin{proof} It follows from Corollary \ref{corollary on action of symn on murphy type elements 2}, part (3),  that $\pi \m  \blam \mfs v {\mft^\blam} 1 = x_1 + x_2$  where $x_2 \in \overline N^\blam$,  and
$x_1$ is a linear combination of terms
$\m \blam  \mfu {v'} {\mfv} 1$ with $\mfu, \mfv$ standard $\blam$--tableaux, 
$\mfv \GammaDominates \mft^\blam$,  $\comp_\mfv = \comp_{\mft^\blam}$, and $v' \in V^\alpha$.    By Remark \ref{remark on dominance and gamma dominance}, we have $\mfv \dominates \mft^\blam$.  Since $\mft^\blam$
is maximum among standard $\blam$--tableaux with respect to dominance order,  
all the terms featuring in the expansion of $x_1$ have $\mfv = \mft^\blam$.  
\end{proof}

In the following discussion, recall  that $\bgamma(\alpha)$ denotes the sequence in $\Gamma^n$ whose first $\alpha_1$ entries are equal to $\gamma\paren 1$,  next  $\alpha_2$ entries are equal to $\gamma\paren 2$, and so forth.

\begin{lemma} \label{lemma: maximality of blame and maximality of gamma alpha}
Let $\blam \in \multipart$ and let $\alpha = \alpha(\blam)$.  If $\blam$ is maximal  in $\multipart$ with respect to $\Gamma$--dominance order $\GammaDominates$, then $\bgamma(\alpha)$ is maximal in $\Gamma^n$ with respect to the product order.
\end{lemma}

\begin{proof} Suppose that $\blam$ is maximal in $\multipart$.   Then for all $\gamma \in \Gamma$, $\blam(\gamma) = \emptyset$ unless $\gamma$ is maximal in $\Gamma$.    For if $\gamma$ is not maximal, and $\blam(\gamma)$ is a non--empty partition,  then there exists a $\gamma' > \gamma$, and one can produce a $\blam' \StrictlyDominates \blam$ by removing a box from the $\gamma$ component of $\blam$ and appending it as a new row to the $\gamma'$ component.  This means that $\alpha_i = 0$ unless
$\gamma(i)$ is maximal, so only maximal $\gamma(i)$'s appear in the sequence $\bgamma(\alpha)$.  Thus
$\bgamma(\alpha)$ is maximal in the product order. 
\end{proof}

\begin{corollary} \label{corollary:  mult by a on left when blam is maximal}
Suppose $\blam \in \multipart$ is maximal with respect to $\GammaDominates$.  Let
$\alpha = \alpha(\blam)$.     Then for all $a \in A^{\otimes n}$,
$
a  y^\alpha
$
is a linear combination of elements $v' y^\alpha$ with $v' \in V^\alpha$.    Similarly, 
$y^\alpha  a$ is a linear combination of elements $y_\alpha (v')^*$ with $v' \in V^\alpha$. 
\end{corollary}

\begin{proof}  Since $A^{\otimes n}$ is cyclic cellular, for any $\bgamma \in \Gamma^n$,  
$$A^{\otimes n} y_\bgamma +  \overline{A^{\otimes n}}^{\,\bgamma} = \spn\{v y_\bgamma: v \in V^\bgamma\}  +  \overline{A^{\otimes n}}^{\,\bgamma}$$  
by Lemma \ref{lemma span of v y gamma}.    If $\bgamma$ is maximal in $\Gamma^n$,  then 
$$A^{\otimes n} y_\bgamma  = \spn\{v y_\bgamma: v \in V^\bgamma\}.$$  
  If $\blam$ is maximal,  then $\bgamma(\alpha)$ is maximal in $\Gamma^n$ by 
Lemma \ref{lemma: maximality of blame and maximality of gamma alpha}.    Moreover, 
$y^\alpha = y_{\bgamma(\alpha)}$ and $V^\alpha = V^{\bgamma(\alpha)}$ by Remark \ref{remark on V alpha}.  Thus,
$$a  y^\alpha \in \spn\{v y^\alpha: v \in V^\alpha\}.$$
   The statement regarding multiplication by $a$ on the right follows similarly.
\end{proof}

\begin{corollary} \label{corollary:  y alpha x lambda in ideal when lambda maximal}
Suppose $\blam \in \multipart$ is maximal with respect to $\GammaDominates$.  Let
$\alpha = \alpha(\blam)$. Then $y^\alpha x_\blam$ is in the span of $\{ v y^\alpha x_\blam w^* : v, w \in V^\alpha\}$. 
\end{corollary}

\begin{proof}   Applying Lemma \ref{lemma span of v y gamma} with to the cyclic cellular algebra $A^{\otimes n}$,  and
$y^\alpha = y_{\bgamma(\alpha)}$ and taking into account that $\bgamma(\alpha)$ is maximal, we get
that $y^\alpha$ is a linear combination of elements $vy^\alpha$ with $v \in V^\alpha$.  Therefore
$y^\alpha x_\blam$ is a linear combination of elements $v y^\alpha x_\blam = v x_\blam y^\alpha$.  
Now apply Lemma \ref{lemma span of v y gamma} again on the right to express each of these elements as a linear combination of elements $v x_\blam y^\alpha w^*$  with $w \in V^\alpha$. 
\end{proof}

The following sets of elements of $A^{\otimes n}$  play an essential role at several points in our arguments.  

\begin{notation}  \label{notation Z alpha}
Let $\alpha$ be a composition of $n$ with $r$ parts.    For $\bgamma > \bgamma(\alpha)$, define 
$V^{\bgamma, \alpha}$ to be the set of simple tensors $v_1 \otimes \cdots \otimes v_n$, where
$$
\begin{aligned}
&v_j = 1  \  &\text{ if }  \ &\gamma_j = \bgamma(\alpha)_j \text{, and}\\
&v_j \in V^{\gamma_j} \  &\text{ if } \  &\gamma_j >   \bgamma(\alpha)_j .
\end{aligned}
$$
\ignore{
We denote by $Z^\alpha$ the span of simple tensors of the form 
$v y_\bgamma w^*$,  where $\bgamma > \bgamma(\alpha)$, $v \in V^\bgamma$, and $w \in V^{\bgamma, \alpha}$.   Let $Z^\alpha_*$ the span of simple tensors of the form 
$v y_\bgamma w^*$,  where $\bgamma > \bgamma(\alpha)$, $v \in V^{\bgamma, \alpha}$, and $w \in V^\bgamma$.
}
\end{notation}

\begin{lemma} \label{special form of cellularity for tensor product algebra}
Let $a  \in A^{\otimes n}$.  Let $\alpha$ be a composition of $n$ with $r$ parts.  Then:
\begin{enumerate}
\item
$a  y^\alpha = z_1 + z_2$,  where $z_1$ is a linear combination of terms of the form
$v' y^\alpha$ with $v' \in V^\alpha$, and $z_2$ is a linear combination of elements
$v y_\bgamma w^*$ with $\bgamma > \bgamma(\alpha)$, $v \in V^\bgamma$, and $w \in V^{\bgamma, \alpha}$.   
\item $  y^\alpha  a = z_1 + z_2$,  where $z_1$ is a linear combination of terms of the form
$y^\alpha (v')^*$ with $v' \in V^\alpha$, and $z_2 $ is a linear combination of elements
$v y_\bgamma w^*$ with $\bgamma > \bgamma(\alpha)$, $v \in V^{\bgamma, \alpha}$, and $w \in V^{\bgamma}$.   

\end{enumerate}
\end{lemma}

\begin{proof}  We can suppose without loss of generality that  $a = a_1 \otimes \cdots \otimes a_n$ is a simple tensor.
Then
$$
a  y^\alpha   = a_1  y_1 \otimes \cdots \otimes a_n y_n,
$$
where  $y_j = y_{\gamma(\alpha)_j}$.
By Lemma \ref{lemma span of v y gamma},  for  each $j$  we have $a_j  y_j = b_1 + b_2$,  where $b_1$ is a linear combination of terms $v'_j y_j$ with $v'_j \in V^{\gamma(\alpha)_j}$,  and 
$b_2$ is a linear combination of terms $v_j y_{\gamma} (w_j )^*$,  where $\gamma> \gamma(\alpha)_j$, and $v_j , w_j  \in V^{\gamma}$.   Statement (1) follows.  The proof of statement (2) is similar.
\end{proof}

\begin{lemma} \label{lemma: y v x lambda in N bar lambda}
Let $\blam \in \multipart$ and let $\alpha = \alpha(\blam)$.   Suppose that for all
$\bmu \StrictlyGammaDominates  \blam$ in $\multipart$,  $N^\bmu$ is a two sided ideal of $A \wr \symn$,  and $y^{\alpha(\bmu)} x_\mu \in N^\mu$.     
Suppose that
$\bgamma  > \bgamma(\alpha)$ in $\Gamma^n$ and 
$v \in V^{\bgamma, \alpha}$.  Then
$y_{\bgamma} v^*  x_\blam \in \overline N^\blam$  and  $x_\blam v y_\bgamma \in \overline N^\blam$.  
\end{lemma}

\begin{proof}  As usual,  write $n_k = \alpha_1 + \cdots + \alpha_{k-1}$ for $1 \le k \le n$.  Write $\bgamma = (\gamma_1, \dots, \gamma_n) $.
 Since $\bgamma > \bgamma(\alpha)$,   
for each $k$ and for $n_k + 1 \le j \le n_{k+1}$,  we have $\gamma_j \ge \gamma(k)$,  with strict inequality for at least one value of $k$ and $j$.  

We partition the set $\{1, 2, \dots, n\}$ into subsets, an integer $j$ being assigned to a subset according to the row of $j$ in $\mft^\blam$ and the value of $\gamma_j$ and $v_j$, as follows:
For integers $k$ and $i$,  let $P_{k, i}$  be the set of integers $j$ such that 
$j$ is in the $i$--th row of the $k$--the component of $\mft^\blam$  and $\gamma_j = \gamma(k)$, so 
$v_j = 1$.
For integers $k$, $i$ and $\ell$,  with $\gamma(\ell) > \gamma(k)$, 
and for $t \in \mathcal T(\gamma\paren\ell)$,  let $P_{k, i, \ell, t}$ be the set of integers $j$ such that
$j$ is in the $i$--th row of the $k$--the component of $\mft^\blam$,  $\gamma_j = \gamma(\ell)$, and
$v_j = v_t$.    
  Then $P= \{P_{k, i}\} \cup  \{P_{k, i, \ell, s}\}$ is a set partition of $\{1, 2, \dots, n\}$.   Moreover,  since $\bgamma > \bgamma(\alpha)$,  it follows that there is at least one index $(k, i, \ell, t)$ with $P_{k, i, \ell, t} \ne \emptyset$.  
Note that $y_{\gamma_j}$ and $  v_j$  are constant for $j$ in each block of the set partition $P$, namely
  $y_{\gamma_j} = y_{\gamma(k)}$ and $v_j = 1$   for $j \in P_{k, i}$;    and
  $y_{\gamma_j} = y_{\gamma(\ell)}$ and  $ v_j = v_t$ for
  $j \in P_{k, i, \ell, t}$.  

Let $\sym_P$ be the set of $w \in \symn$ such that $w(X) = X$ for each block $X \in P$.
Then $\sym_P$ is a subgroup of  $\sym_\blam$.   Let $x_P$ denote the corresponding symmetrizer, 
$x_P = \sum_{w \in \sym_P}  w$.  We have $x_\blam =  x_P S$,  where $S$ is the sum of the elements in a complete family of right coset representatives of $\sym_P$ in $\sym_\blam$.

Note that $x_P$ commutes with $v^*$, so that
$$
 y_{\bgamma} v^*  x_\blam = 
 y_{\bgamma} v^*  x_P S =
 (y_{\bgamma} x_P)  v^* S.
$$
We will now show that $y_\bgamma x_P \in N^\bnu$ for some $\bnu \StrictlyGammaDominates \blam$.

Let $\bnu$  be the multipartition such that for each $\ell$,  $\nu\power \ell$ has parts 
$\{|P_{\ell, i}|\}_{i} \cup \{|P_{k, i, \ell, t}|\}_{k, i, t}$,  sorted into decreasing order.  
Then  $\bnu \StrictlyGammaDominates \blam$, because $\bnu$ can be obtained from $\blam$ by applying ``raising operators".    In fact, $\bnu$ can be obtained from $\blam$ in two stages.  First,
for each pair $k, \ell$ with $\gamma\paren\ell > \gamma\paren k$, and for each $i$ and $t$
such that  $P_{k, i, \ell, t}$ is non-empty,  move $|P_{k, i, \ell, t}|$ boxes from the $i$--th row of $\lambda\power k$ to form a new row appended to $\lambda\power \ell$.  
Then sort the rows of each component of the resulting multicomposition into decreasing order; this involves raising boxes from lower rows to higher rows in each component.  Note that 
$\bnu \ne \blam$ because some $P_{k, i, \ell, t}$ is non-empty. 

There exists $\pi \in \symn$ such that
$y_\bgamma x_P  = \pi y^{\alpha(\bnu)} x_\bnu  \pi\inv$.    In fact, $\pi$ can be chosen so that $\pi\inv$ sends the elements of each non-empty $P_{k,i}$ to the entries in some row of the $k$--th component of $t^\bnu$ and the elements of each non-empty $P_{k,i,\ell,t}$ to the entries in some row of the $\ell$--th component of $t^\bnu$.  
Since $N^\bnu$ is assumed to be a two sided ideal of $A\wr \symn$, and $y^{\alpha(\bnu)} x_\bnu$ is assumed to be an element of $N^\bnu$,   it follows that
 $y_\bgamma x_P  \in N^\bnu$, and therefore
$  y_{\bgamma} v^*  x_\blam =
 (y_{\bgamma} x_P)  v^* S \in N^\bnu$.  But
$N^\bnu \subset \overline N^\blam$, since $\bnu \StrictlyGammaDominates \blam$.     This completes the proof that $y_{\bgamma} v^*  x_\blam \in \overline N^\blam$.   The proof that  $x_\blam v y_\bgamma \in \overline N^\blam$ is similar.
 \end{proof}

\begin{proposition} \label{proposition on multiplication of murphy type elements by A}
Let $\blam \in \multipart$, and let $\alpha = \alpha(\blam)$.  Let $\mfs, \mft$ be standard $\blam$--tableaux, and let $a \in A^{\otimes n}$.
\begin{enumerate}
\item  For  $v \in V^{\alpha}$, and $w \in V^\alpha_1$,  
$a \, \m \blam \mfs v \mft w = x_1 + x_2$,
where 
$x_1$  is a  linear combination of terms
$\m \blam {\mfs} {v'} {\mft} {w}$, with 
$v' \in V^{\alpha}$, and
$x_2 \in \overline{N}^\blam$.
\item  For  $v, w \in V^{\alpha}$, $  \m \blam \mfs v \mft w\, a = x_1 + x_2$,
where 
$x_1$  is a  linear combination of terms
$\m \blam {\mfs} {v} {\mft} {w'}$, with 
$w' \in V^{\alpha}$, and
$x_2 \in \overline{N}^\blam$.
  \item $N^\blam$ and $\overline{N}^\blam$  are two sided ideals of $A\wr \symn$. 
  \item  $y^\alpha x_\blam \in N^\blam$.  
\end{enumerate}
\end{proposition}

\begin{proof}  We prove all  parts together, by induction on the partial order $\GammaDominates$ of $\multipart$.     For parts (1) and (2), we can assume without loss of generality that $a$ is a simple tensor,
$a = a_1 \otimes \cdots \otimes a_n$.   

First assume  that $\blam$
is maximal in $(\multipart, \GammaDominates)$. 
For $a \in A^{\otimes n}$,  we have
$$
\begin{aligned}
a \m \blam \mfs v \mft w &= a d(\mfs) v y^{\alpha} x_\blam w^* d(\mft)^* \\
&= d(\mfs) \,(\leftexp{d(\mfs)^*}{\hskip -.3 em a })v y^{\alpha} x_\blam w^* d(\mft)^*.
\end{aligned}
$$
By Corollary \ref{corollary:  mult by a on left when blam is maximal}, $(\leftexp{d(\mfs)^*}{\hskip -.3 em a })v y^{\alpha} $ is a linear combination of elements
$v' y^\alpha$ with $v' \in V^\alpha$, since $\blam$ is maximal.  Hence
$a \m \blam \mfs v \mft w$ is a linear combination of terms of the form
$\m \blam \mfs {v'} \mft w$.
This demonstrates statement (1) of the Proposition in case $\blam$ is maximal.  Part (2) follows similarly. 
We  have shown that  
 $N^\blam$ is invariant under multiplication from the left  or right by 
elements $a \in A^{\otimes n}$.  Corollary \ref{corollary on action of symn on murphy 
type elements 2}  implies that $N^\blam$ is invariant under multiplication from the left  or right by 
elements of $\symn$.   Since $A \wr \symn$ is generated as an algebra by $A^{\otimes n}$ and $\symn$, we have that $N^\blam$ is a two sided  ideal in $A \wr \symn$.   Since $\blam$ is maximal, $\overline N^\blam = (0)$.  
 Finally, when $\blam$ is maximal,  $y^\alpha x_\blam \in N^\blam$ by Corollary \ref{corollary:  y alpha x lambda in ideal when lambda maximal}.    This completes the proof of all the statements under the assumption that $\blam$ is maximal.

Now take $\blam \in \multipart$ arbitrary, and assume that all the statements  of the Proposition 
holds when $\blam$ is replaced by $\bmu \in \multipart$ with 
$\bmu \StrictlyGammaDominates \blam$.  
Since $\overline{N}^\blam$ is the span of 
$\{N^{\bmu} : \bmu \StrictlyGammaDominates \blam\}$,  
it already follows that $\overline{N}^\blam$ is a two sided ideal.
As before, we have
$$
\begin{aligned}
a \m \blam \mfs v \mft w 
&= d(\mfs) \,(\leftexp{d(\mfs)^*}{\hskip -.3 em a })v y^{\alpha} x_\blam w^* d(\mft)^*.
\end{aligned}
$$
Write
$
(\leftexp{d(\mfs)^*}{\hskip -.3 em a }) = a'$.
Apply Lemma \ref{special form of cellularity for tensor product algebra} (1) to $(a' v )y^\alpha$.  
It follows that
$a' v y^\alpha = z_1 + z_2$,  where $z_1$ is a linear combination of terms of the form
$v' y^\alpha$ with $v' \in V^\alpha$, and $z_2$  is a linear combination of elements
$v y_\bgamma w^*$ with $\bgamma > \bgamma(\alpha)$, $v \in V^\bgamma$, and $w \in V^{\bgamma, \alpha}$. 
Since $N^\bmu$ is a two sided ideal 
and $y^{\alpha(\bmu)} x_\bmu \in N^\bmu$ for all $\bmu \StrictlyGammaDominates \blam$, 
by the inductive hypothesis, we can apply Lemma \ref{lemma: y v x lambda in N bar lambda}, which tells us that
$z_2 x_\blam \in \overline N^\blam$.  On the other hand, $d(\mfs) z_1 x_\blam w^* d(\mft)^*$ is a linear combination of terms of the form $\m \blam \mfs {v'} \mft w$.
This proves
statement (1).  Statement (2) is proven similarly, using  
 Lemma \ref{special form of cellularity for tensor product algebra} (2)

This finishes the proof of parts (1)  and (2) of the Proposition.  
 We have already seen, in the discussion of the case that $\blam$ is maximal, that 
parts (1) and (2) of the Proposition together with 
Corollary \ref{corollary on action of symn on murphy type elements 2}  
implies that $N^\blam$ is a two sided ideal.  

Finally, we prove  part (4),   By Lemma \ref{special form of cellularity for tensor product algebra} part (1), we have
$y^\alpha = z_1 + z_2$,  where $z_1$ is a linear combination of terms of the form
$v y^\alpha$ with $v \in V^\alpha$, and $z_2$ is a linear combination of elements
$v y_\bgamma w^*$ with $\bgamma > \bgamma(\alpha)$, $v \in V^\bgamma$, and $w \in V^{\bgamma, \alpha}$.  Since by the induction hypothesis, $y^{\alpha(\bmu)} x^\bmu \in N^\bmu$ for all $\bmu \StrictlyDominates \blam$, and since we know at this point that $N^\bmu$ is an ideal for all $\bmu$, we can apply 
Lemma \ref{lemma: y v x lambda in N bar lambda} to obtain that $z_2 x_\blam \in \overline N^\blam$.  
Thus $y^\alpha x_\blam$ is congruent modulo $\overline N^\blam$ to a linear combination of terms
$v y^\alpha x_\blam = v x_\blam y^\alpha$,  with $v \in V^\alpha$.   Applying 
Lemma \ref{special form of cellularity for tensor product algebra} part (2) and Lemma \ref{lemma: y v x lambda in N bar lambda} again to each of these terms gives that $y^\alpha x_\blam$ is congruent modulo $\overline N^\blam$ to a linear combination of terms $v y^\alpha x_\blam w^*$ with $v, w \in V^\alpha$.  But $vy^\alpha x_\blam w^* = \m  \blam  {\mft^\blam} v {\mft^\blam} w \in N^\blam$. 
\end{proof}

\begin{corollary} \label{corollary: y alpha x lambda in span of v y alpha x lambda}
Let $\blam \in \multipart$,  and write $\alpha = \alpha(\blam)$. We have
$$y^\alpha x_\blam \in
\spn\{ v y^\alpha x_\blam: v \in V^\alpha \} + \overline N^\blam.$$
\end{corollary}

\begin{proof}  The  proof of Proposition \ref{proposition on multiplication of murphy type elements by A} part (4) shows that  $y^\alpha x_\blam $ is congruent modulo $\overline N^\blam$ to a linear combination of terms $v y^\alpha x_\blam $ where $v \in V^\alpha$.  
\end{proof}

\begin{proposition} \label{proposition: multiplicative property of general murphy basis elements}
Suppose that $\blam \in \multipart$, 
and $(\mfs, v) \in \mathcal T(\blam)$.
   Let $x \in A\wr \symn$.  Then for each $(\mfu, v') \in \mathcal T(\blam)$,  there exists $r_{(\mfu, v')} \in R$ such that
\begin{equation}\label{proposition: multiplicative property of general murphy basis elements equation 1}
x \m \blam \mfs v {\mft^\blam} 1 \equiv \sum_{(\mfu, v') \in \mathcal T(\blam)} r_{(\mfu, v')} \, \m \blam \mfu {v'} {\mft^\blam} 1\mod \overline N^\blam.
\end{equation}   
 Moreover,   
for all $(\mft, w) \in \mathcal T(\blam)$, 
\begin{equation}\label{proposition: multiplicative property of general murphy basis elements equation 2}
x \m \blam \mfs v \mft w \equiv \sum_{(\mfu, v') \in \mathcal T(\blam)} r_{(\mfu, v')} \, \m \blam \mfu {v'} \mft w \mod \overline N^\blam.
\end{equation}
\end{proposition}

\begin{proof}   Let $U$ be the $R$--module spanned by the set of 
$\m \blam  {\mfs'} {v'}  {\mft^\blam} 1$,   where $(\mfs', v') \in \mathcal T(\blam)$.  
By Proposition \ref{proposition on multiplication of murphy type elements by A} (1)  and Corollary \ref{final lemma on action of symn on murphy type elements} ,  $U + \overline N^\blam$ is invariant under left multiplication by elements $a \in A^{\otimes n}$ and by $\pi \in \symn$.  
Since these elements generate $A \wr \symn$ as an algebra,
it follows that $U + \overline N^\blam$ is a left ideal.  Therefore, there exist coefficients
$r_{(\mfu, v')} \in R$ such that Equation \eqref{proposition: multiplicative property of general murphy basis elements equation 1} holds.
Now for any  $(\mft, w) \in \mathcal T(\blam)$,  multiply the congruence  of Equation 
 \eqref{proposition: multiplicative property of general murphy basis elements equation 1} 
on the right by
$w^* d(\mft)^*$.  Since $\overline N^\blam$ is an ideal and
$$
\m \blam \mfu {v'} {\mft^\blam} 1 w^* d(t)^* = \m \blam \mfu {v'} \mft w,
$$
this yields the congruence of Equation  \eqref{proposition: multiplicative property of general murphy basis elements equation 2} .
\end{proof}

\begin{lemma} \label{lemma involution on general murphy type elements} 
Let $\blam \in \multipart$ and  write $\alpha = \alpha(\blam)$. Let $(\mfs, v), (\mft, w) \in \mathcal T(\blam)$.
\begin{enumerate}
\item  $(y^\alpha x_\blam)^* \equiv y^{\alpha} x_\blam \mod \overline N^\blam$. 
\item
$(\m \blam \mfs v \mft w)^*   \equiv \m \blam \mft w \mfs v \mod \overline N^\blam$.
\end{enumerate}
\end{lemma}

\begin{proof} Write $y^\alpha = y_1 \otimes \cdots y_n$, where for each $j$,  $y_j = y_{\bgamma(\alpha)_j}$.  By point 2(a) of Lemma \ref{lemma: equivalent conditions for cyclic cellular algebra},   we have that $y_j^*$ is equal to $y_j$ plus a linear combination of elements of the form $v_j y_{\gamma'} w_j$, where $\gamma' > \bgamma(\alpha)_j$, and $v_j, w_j \in V^{\gamma'}$.  Therefore
$(y^\alpha)^*$ is equal to $y^\alpha + z$, where $z$ is a linear combination of terms
$v y_\bgamma w^*$,  where $\bgamma > \bgamma(\alpha)$  and both $v$ and $w$ are in
$V^{\bgamma, \alpha}$.  Since $x_\blam^* = x_\blam$ and $x_\blam$ commutes with $y^\alpha$, 
$$
(y^\alpha x_\blam)^* =  (y^\alpha)^* x_\blam    = y^\alpha x_\blam + z x_\blam.
$$
Since $N^\bmu$ is  a two sided ideal and $y^{\alpha(\bmu)} x_\bmu \in N^\bmu$, 
for all 
$\bmu \in \multipart$,  
we can apply 
Lemma \ref{lemma: y v x lambda in N bar lambda}, 
which tells us that $z x_\blam \in \overline N^\blam$.   This proves part (1).
Part (2) follows from part (1) together with the factorization 
$$
\m \blam \mfs v \mft w = d(\mfs) v y^\alpha x_\blam w^* d(\mft)^*.
$$
\end{proof}

We are now ready to propose a cellular structure for the wreath product $A \wr \symn$.  The partially ordered set for the cellular structure is $(\multipart, \GammaDominates)$.  For each 
$\blam \in \multipart$, we take the index set $\mathcal T(\blam)$ to be one introduced above, namely, the set of pairs $(\mfs, v)$,  where $\mfs$ is a standard $\blam$--tableau and $v \in V^{\alpha(\blam)}$.   The proposed cellular basis is 
$$
\mathcal B = \{\m \blam \mfs v \mft w :  \blam \in \multipart,  (\mfs, v), (\mft, w) \in \mathcal T(\blam)\}.
$$
For $\blam \in \multipart$, we set 
$$
y_\blam =  y^{\alpha(\blam)} x_\blam.
$$
For $\blam \in \multipart$ and $(\mfs, u) \in \mathcal T(\blam)$, we set
$$
v_{(\mfs, u)} =  d(\mfs) u.
$$
Thus,
$$
\m \blam \mfs u \mft w  = v_{(\mfs, u)}y_\blam ( v_{(\mft, w)})^*.
$$

We have already verified that $\mathcal B$ satisfies several of the defining properties of a cellular basis.  In fact, Proposition \ref{proposition: multiplicative property of general murphy basis elements} and 
 Lemma \ref{lemma involution on general murphy type elements}
give us properties (2) and (3) of Definition \ref{weak cellularity}. 
What remains to show is that $\mathcal B$ is an $R$--basis of $A \wr \symn$.  

\begin{lemma}  $\mathcal B$ spans $A \wr \symn$ over $R$.  
\end{lemma}

\begin{proof}  Let $U$ denote the $R$--span of $\mathcal B$.
It follows from Proposition \ref{proposition: multiplicative property of general murphy basis elements} and 
 Lemma \ref{lemma involution on general murphy type elements} 
that $U$ is a two sided ideal in $A \wr \symn$.  
Fix a composition $\alpha$ of $n$ and take 
$$
\blam = ( (1)^{\alpha_1}, (1)^{\alpha_2}, \dots, (1)^{\alpha_r}).
$$
Then $x_\blam = 1$ and for $v, w \in V^\alpha$, 
$$
\m \blam {\mft^\blam} v  {\mft^\blam} w = v y^\alpha w^*.
$$
For $\pi \in \symn$, 
$$
\leftexp \pi {\hskip - .1 em (v y^\alpha w^*)} = \pi v y^\alpha w^* \pi\inv  \in U,
$$
since $U$ is a two sided ideal.

The collection of these elements, as $\alpha$, $v$, $w$, and $\pi$  vary,  constitute a (cellular) basis of $A^{\otimes n} \subseteq A \wr \symn$. In particular, the multiplicative identity
$1$ is contained in $U$.
\end{proof}

To show that $\mathcal B$ is linearly independent over $R$, we require an enumerative result. 

Let $d_i$ denote the rank of the cell module of $A$ associated with $\gamma_i \in \Gamma$, 
$d_i = |\mathcal T(\gamma_i)|$.  Let $d = \sum_{i = 1}^r d_i^2$ denote the rank of $A$.  

For $\blam \in \multipart$, let $f_\blam$ denote the number of standard $\blam$--tableaux;  then we have
$$
f_\blam =  \binom {n}{\alpha(\blam)}  \prod_i f_{\lambda\power i},
$$
where $\binom {n}{\alpha(\blam)}  $ denotes the multinomial coefficient 
$\binom {n} {|\la \power 1|, \dots, |\la \power r|}$.   The size of $\mathcal T(\blam)$ is
$$
|\mathcal T(\blam)| = |V^{\alpha(\blam)}| f_\blam = \left (\prod_i d_i^{|\la \power i|} \right ) f_\blam.
$$

\begin{claim} \label{enumerative claim}
 $$\sum_{\blam \in \multipart} |\mathcal T(\blam)|^2 =  d^n n !.$$
\end{claim}

To verify this claim, we will use some results about wreath products of semisimple algebras.
Let $B$ be a semisimple algebra over the complex numbers $\C$, with simple modules $W_1, \dots, W_r$  of dimensions $\dim(W_i) = d_i$;  hence $\dim(B) = \sum_i d_i^2 = d$.    The simple modules of $B \wr \symn$  are indexed by multipartitions
$\blam = (\la\power 1, \dots, \la\power r)$.  For a multipartition $\blam$, let $\alpha = \alpha(\blam)$. 
Let $W^\alpha$ denote the simple $B^{\otimes n}$ module
$$
W^\alpha = W_1^{\otimes \alpha_1} \otimes W_2^{\otimes \alpha_2} \otimes \cdots \otimes W_r^{\otimes \alpha_r}
$$
and let $S^\blam$ denote the simple $\sym_\alpha$ module
$$
S^\blam = S^{\la \power 1} \otimes \cdots \otimes S^{\la \power r},
$$
where for a partition $\la$ of size $k$, $S^\la$ is the Specht module of $\sym_k$ corresponding to $\la$.  
The simple    $B \wr \symn$--module corresponding to $\blam$ has the form
$$
W^\blam = \Ind_{\sym_\alpha}^{\symn} (W^\alpha \otimes  S^\blam).$$
An explanation of this construction of simple modules of $B \wr  \symn$ can be found in ~\cite{macdonald-wreath} and ~\cite{chuang-meng}.
The dimension of the module  $W^\blam$ is
$$
\dim(W^\blam) = 
[\symn: \sym_\alpha]  \dim(W^\alpha)\dim(S^\blam) =  \binom n \alpha \prod_i d_i^{|\la \power i|} \prod_i f_{\la \power i} = |\mathcal T(\blam)|. 
$$
Since $\dim(B \wr \symn) = d^n n!$,  Claim \ref{enumerative claim} is  equivalent to the assertion that $B \wr \symn$ is semisimple.  This follows from a more general Maschke-type lemma:

\begin{lemma} If $B$ is a semisimple algebra over a field of characteristic zero and $G$ is a finite group
acting on $B$ by automorphisms, then $B \rtimes G$ is semisimple.
\end{lemma}

\def\Tr{{\rm Tr}}
\def\rad{{\rm rad}}
\noindent
{\em Sketch of a proof:}  One can reduce to the case that the field is algebraically closed by tensoring with a field extension.   The trace  $\Tr$ of the left regular representation of $B \rtimes G$ satisfies
$\Tr(\sum_{g\in G} b_g g)  = \Tr(b_e) = |G|\,  \Tr_B(b_e)$, where $\Tr_B$ is the trace of the left regular representation of $B$.    It follows from this that $\Tr$ is non-degenerate, and therefore   $\rad(B \rtimes G)$ is zero, since  $\rad(B \rtimes G) \subseteq \rad(\Tr)$. \qed

\medskip
Returning to our candidate $\mathcal B$ for a cellular basis of $A \wr \symn$, we have:
\begin{lemma}  $\mathcal B$ is linearly independent over $R$.  
\end{lemma}

\begin{proof}  We already know that $\mathcal B$ spans $A \wr \symn$ over $R$, and the cardinality of 
$\mathcal B$ is $\sum_\blam |T(\blam)|^2$.  By Claim \ref{enumerative claim}, this equals $d^n n!$. On the other hand,
$A \wr \symn$ is free of rank $d^n n!$ over $R$ since $A$ is free of rank $d$.   It follows that $\mathcal B$ is linearly independent, since the ground ring $R$ is assumed to be an integral domain. 
\end{proof}

\noindent
{\em Completion of the proof of Theorem \ref{theorem cellularity of wreath products}}. We have shown that $\mathcal B$ is a cellular basis of $A \wr \symn$.  Fix $\blam \in  \multipart$ and let $\alpha = \alpha(\blam)$.  
Proposition \ref{proposition: multiplicative property of general murphy basis elements}
 shows that
$$
\begin{aligned}
M &:= \spn\{\m \blam  {\mfs} {u}  {\mft^\blam} 1  + \overline N^\blam:  (\mfs, u) \in \mathcal T(\blam) \}  \\ &= \spn\{ v_{(\mfs, u)} \,y_\blam +  \overline N^\blam :  (\mfs, u) \in \mathcal T(\blam)\}
\end{aligned}
$$
is an $A\wr \symn$ module and isomorphic to the cell module $\Delta^\blam$.  
Corollary \ref{corollary: y alpha x lambda in span of v y alpha x lambda} implies that
$y_\blam + \overline N^\blam \in M$, so that $\Delta^\blam$ is cyclic.  \qed

\subsection{The structure of cell modules of $A \wr \symn$}
Let $A$ be any $R$--algebra,  and let $E_1, \dots, E_r$ be a collection of $A$--modules.  We discuss a construction of $A \wr \symn$--modules based on this family of $A$--modules.  For a thorough discussion of this construction, see ~\cite{chuang-meng}.   

Let $\blam = (\la\power 1, \dots, \la\power r)$ be a multipartition of total size $n$ with $r$ parts.  Let $\alpha = \alpha(\blam)$.   Recall from Section 
\ref{subsection the cellular structure of the group algebra of a Young subgroup} that
$$
\Delta^{\blam}_R =  \Delta^{\lambda\power 1}_R \otimes \cdots \otimes  \Delta^{\lambda\power r}_R,
$$
is a cell module for $R \sym_\alpha \cong R\sym_{\alpha_1} \otimes \cdots \otimes R\sym_{\alpha_r} $, and that we identify $\Delta^{\blam}_R$ with the $R$--span of the set of $d(\mfs) x_\blam + J^\blam$,  where $\mfs$ is a standard $\blam$--tableau of the initial kind, and  $J^\blam$ denotes
$\overline{R\sym_\alpha}^\blam$.

Let $E^\alpha = E_1^{\otimes \alpha_1} \otimes \cdots \otimes E_r^{\otimes \alpha_r}$.    Then $E^\alpha$ is an $A \wr \sym_\alpha$--module, with $A^{\otimes n}$ acting by the tensor product action and
$\sym_\alpha$ acting by place permutations.  Moreover, $E^\alpha \otimes \Delta^\blam_R$ is also
an $A \wr \sym_\alpha$--module, with $a (v \otimes m) = a v \otimes m$ and
$\pi( v \otimes m) = \leftexp{\pi}{\hskip - .1 em v} \otimes \pi m$,  for $a \in A^{\otimes n}$,  $\pi \in \sym_\alpha$,  $v \in E^\alpha$ and $m \in \Delta^\blam_R$.  

Finally, we obtain an $A \wr \symn$--module by
$$
\Ind_{A\wr\sym_\alpha}^{A\wr \symn}(E^\alpha \otimes \Delta^\blam_R)
= (A\wr \symn) \otimes_{A \wr \sym_\alpha} (E^\alpha \otimes  \Delta^\blam_R) .
$$

When this construction is applied to the simple modules of $A$, one obtains the simple modules of the wreath product $A \wr \symn$.   More exactly,  suppose now that $A$ is a finite dimensional algebra over a field $K$ and that $K$ is a splitting field for $A$.  Let $E_1, \dots, E_r$ be a complete list of mutually inequivalent simple $A$--modules.  Assume moreover that $\char(K) = 0$, so that  $K\sym_j$ is split semisimple for all $j \ge 0$, and the cell modules $\Delta^\lambda_K$, for $\la \in \Lambda_j$,  are the simple $K\sym_j$--modules.

The following theorem is obtained by adapting the analysis of the representation theory of wreath products of groups ~\cite{Specht-1933}, and of semidirect products of groups in general ~\cite{Clifford-1937}
 to wreath product of algebras.  A proof was given  in ~\cite{macdonald-wreath}.  See also ~\cite{chuang-meng}.

\begin{theorem}  Let $A$ be a finite dimensional algebra over a field $K$ of characteristic zero, with $K$ a splitting field for $A$.  Let $E_1, \dots, E_r$ be a complete list of simple $A$--modules, up to isomorphism.   Then the modules
$$
\Ind_{A\wr\sym_\alpha}^{A\wr \symn}(E^\alpha \otimes \Delta^\blam_K),
$$
where $\blam$ varies over multipartitions of total size $n$ with $r$ parts and $\alpha = \alpha(\blam)$,
are simple and mutually inequivalent.  All simple $A\wr \symn$--modules are of this form.
\end{theorem}

Not surprisingly,  when $A$ is a cyclic cellular algebra over an integral domain $R$,  the cell modules of the wreath product $A \wr \symn$ are obtained by the same construction.  Let $A$ be such an algebra, and adopt the notations of Section \ref{subsection: cellularity theorem for wreath products} regarding the cellular structure of $A$ and of $A \wr \symn$.  
In particular $(\Gamma, \ge)$   is the partially ordered set associated with the cellular structure of $A$, and  $\Gamma = (\gamma(1), \dots, \gamma(r))$ is a listing of $\Gamma$ consistent with the partial order.  In order to avoid confusing overuse of the symbol $\Delta$ for the cell modules of various algebras,  let $E_i$ denote the cell module of $A$ associated to $\gamma(i) \in \Gamma$, and let
$C^\blam$ denote the cell module of $A\wr \symn$ associated to $\blam \in \multipart$.  
Recall that $\bgamma(\alpha) \in \Gamma^n$ is the sequence in $\Gamma^n$  whose first $\alpha_1$ entries are equal to $\gamma(1)$,  next $\alpha_2$ entries are equal to $\gamma(2)$, and so forth.
The $A^{\otimes n}$--module $E^\alpha = E_1^{\otimes \alpha_1} \otimes \cdots \otimes E_r^{\otimes \alpha_r}$ is the cell module corresponding to $\bgamma(\alpha)$.  We are going to write
$I^\alpha$ for $\overline{A^{\otimes n}}^{\,\gamma(\alpha)}$, namely the span of elements of the cellular basis of $A^{\otimes n}$ corresponding to $\bgamma \in \Gamma^n$ with $\bgamma > \gamma(\alpha)$.   The cell module $E^\alpha$ has basis $\{v y^\alpha +  I^\alpha:  v \in V^\alpha\}$.  

\begin{lemma}  \label{lemma y alpha J lambda contained in N bar lambda}
Let $\blam \in \multipart$ and let $\alpha = \alpha(\blam)$.  Write 
$J^\blam$ for $\overline{R\sym_\alpha}^\blam$.  We have $y^\alpha J^\blam \subset \overline N^\blam$.
\end{lemma}

\begin{proof}  $J^\blam$ is the span of elements  $d(\mfs) x_\bmu d(\mft)^*$  where
$\bmu \StrictlyDominates \blam$,  $\alpha(\bmu) = \alpha$,  and $\mfs, \mft$ are standard $\bmu$--tableaux of the initial kind.  In particular $d(\mfs), d(\mft) \in \sym_\alpha$.   We have
$y^\alpha d(\mfs) x_\bmu d(\mft)^* = d(\mfs) y^\alpha x_\bmu d(\mft)^*$.  But $y^\alpha x_\bmu \in N^\bmu \subseteq \overline N^\blam$ by Proposition \ref{proposition on multiplication of murphy type elements by A} (4).  
\end{proof}

\begin{theorem}  Let $\blam \in \multipart$ and let $\alpha = \alpha(\blam)$.  The cell module
$C^\blam$ of $A \wr \symn$ satisfies
$$C^\blam \cong \Ind_{A\wr\sym_\alpha}^{A\wr \symn}(E^\alpha \otimes \Delta^\blam_R).$$
\end{theorem}

\begin{proof}    
Note that
$$
\Ind_{A\wr\sym_\alpha}^{A\wr \symn}(E^\alpha \otimes \Delta^\blam_R) = 
\bigoplus_\omega  \omega \otimes E^\alpha \otimes \Delta^\blam_R,
$$
where $\omega$ varies through the set of shortest left coset representatives of $\sym_\alpha$ in $\symn$.  
Thus the set of elements
\begin{equation} \label{basis element of Ind E alpha Delta lambda}
\omega \otimes (v y^\alpha + I^\alpha) \otimes (d(\mfs) x_\blam + J^\blam)
\end{equation}
where $\omega$ is such a coset representative, $v \in V^\alpha$, and $\mfs$ is a standard $\blam$--tableau of the initial kind, is a basis of $\Ind_{A\wr\sym_\alpha}^{A\wr \symn}(E^\alpha \otimes \Delta^\blam_R)$.   Note that $v y^\alpha$ is a the unique lifting of $vy^\alpha + I^\alpha$ to an element of the cellular basis of $A^{\otimes n}$ and likewise $d(\mfs) x_\blam$ is the unique lifting of 
$d(\mfs) x_\blam + J^\blam$ to an element of the cellular basis of $R \sym_\alpha$.
Therefore we have a well defined $R$--linear map
$$ \varphi: \Ind_{A\wr\sym_\alpha}^{A\wr \symn}(E^\alpha \otimes \Delta^\blam_R) 
\to  (A \wr \symn)/\overline N^\blam
$$ 
determined by 
$$
\varphi(\omega \otimes (v y^\alpha + I^\alpha) \otimes (d(\mfs) x_\blam + J^\blam)) = 
\omega vy^\alpha d(\mfs) x_\blam + \overline N^\blam.
$$
We have
$$
\omega vy^\alpha d(\mfs) x_\blam + \overline N^\blam =
\omega d(\mfs) \, (\leftexp{d(\mfs)^*} {\hskip -.2 em v}) y^\alpha  x_\blam + \overline N^\blam.
$$
Moreover, $\omega d(\mfs) = d(\omega s)$, the $\blam$--tableau $\omega \mfs$ is   standard, and all standard $\blam$-tableaux  can be written uniquely as  $\omega \mfs$,  where $\omega$ is a distinguished left coset representative of $\sym_\alpha$ in $\symn$ and $\mfs$ is  a standard $\blam$ tableau of the initial kind.   Thus $\varphi$ maps the basis of $\Ind_{A\wr\sym_\alpha}^{A\wr \symn}(E^\alpha \otimes \Delta^\blam_R) $ onto a basis of $C^\blam \subset  (A \wr \symn)/\overline N^\blam$, and thus $\varphi$ is a linear isomorphism with range
$C^\blam$.

It remains to show that $\varphi$ is an $A\wr \symn$--module map.  If $\pi \in \symn$, 
$\pi \,\omega$ can be written uniquely as $\omega' \,\pi'$  for $\omega'$ a distinguished left coset representative of $\sym_\alpha$ in $\symn$,  and $\pi' \in \sym_\alpha$.   Then we have
\begin{equation*} \label{equation check A wreath S n module map 1}
\pi(\omega \otimes (v y^\alpha + I^\alpha) \otimes (d(\mfs) x_\blam + J^\blam)) =
\omega' \otimes (\leftexp{\pi'}{\hskip - .2 em v} y^\alpha + I^\alpha) \otimes \pi' ( d(\mfs) x_\blam + J^\blam)
\end{equation*}
There exist coefficients $r^{\mfs}_\mft \in R$ and an element $h \in J^\blam$  such that
\begin{equation*} \label{equation check A wreath S n module map 2}
\pi' d(\mfs) x_\blam  = \sum_\mft r^{\mfs}_\mft  \, d(\mft) x_\blam + h.
\end{equation*}
Thus
\begin{equation*} \label{equation check A wreath S n module map 3}
\pi(\omega \otimes (v y^\alpha + I^\alpha) \otimes (d(\mfs) x_\blam + J^\blam)) =
\sum_\mft  r^{\mfs}_\mft  (\omega' \otimes (\leftexp{\pi'}{\hskip - .2 em v} y^\alpha + I^\alpha) \otimes (d(\mft) x_\blam + J^\blam)),
\end{equation*}
and
\begin{equation} \label{equation check A wreath S n module map 4}
\varphi(\pi(\omega \otimes (v y^\alpha + I^\alpha) \otimes (d(\mfs) x_\blam + J^\blam)) ) =  \sum_\mft  r^{\mfs}_\mft  \, \omega' (\leftexp{\pi'}{\hskip - .2 em v}) y^\alpha d(\mft) x_\blam + \overline N^\blam.
\end{equation}
On the other hand,
\begin{equation} \label{equation check A wreath S n module map 5}
\begin{aligned}
\pi \varphi(\omega \otimes (v y^\alpha + I^\alpha) \otimes (d(\mfs) x_\blam + J^\blam))
&= \pi  \omega vy^\alpha d(\mfs) x_\blam + \overline N^\blam \\
&= \omega' \pi'  vy^\alpha d(\mfs) x_\blam + \overline N^\blam \\
&=  \omega' (\leftexp{\pi'}{\hskip - .2 em v}) y^\alpha \pi' d(\mfs) x_\blam + \overline N^\blam \\
&=      \omega' (\leftexp{\pi'}{\hskip - .2 em v}) y^\alpha   (\sum_\mft r^{\mfs}_\mft  \, d(\mft) x_\blam + h) + \overline N^\blam \\
&=  \sum_\mft  r^{\mfs}_\mft  \, \omega' (\leftexp{\pi'}{\hskip - .2 em v}) y^\alpha d(\mft) x_\blam + \overline N^\blam,
\end{aligned}
\end{equation}
because $y^\alpha h \in  \overline N^\blam$, by Lemma \ref{lemma y alpha J lambda contained in N bar lambda}. 

Next, for $a \in A^{\otimes n}$, let $a' =  \leftexp{\omega^*}{\hskip -.2 em a}$.  We have
\begin{equation} \label{equation check A wreath S n module map 6}
a(\omega \otimes (v y^\alpha + I^\alpha) \otimes (d(\mfs) x_\blam + J^\blam)) =
\omega \otimes  a'( v y^\alpha + I^\alpha) \otimes ( d(\mfs) x_\blam + J^\blam).
\end{equation}
In the following, let $Z^\alpha$ denote the span of elements $v' y_\bgamma w^*$, where 
$\bgamma > \bgamma(\alpha)$, $v' \in V^\bgamma$, and $w \in V^{\bgamma, \alpha}$.  ($V^{\bgamma, \alpha}$ was defined in Notation \ref{notation Z alpha}.  Note that $Z^\alpha \subseteq I^\alpha$.) 
By Lemma \ref{special form of cellularity for tensor product algebra},  there exist coefficients $r^v_{v'} \in R$ and $z \in Z^\alpha$ such that
$a' v y^\alpha = \sum_{v'} r^v_{v'}\, v' y^\alpha + z$.       From Equation \refc{equation check A wreath S n module map 6}, we have
\begin{equation*} \label{equation check A wreath S n module map 7}
a(\omega \otimes (v y^\alpha + I^\alpha) \otimes (d(\mfs) x_\blam + J^\blam)) =
\sum_{v'} r^v_{v'} \, (\omega \otimes  ( v' y^\alpha + I^\alpha) \otimes ( d(\mfs) x_\blam + J^\blam)).
\end{equation*}
Therefore,
\begin{equation} \label{equation check A wreath S n module map 8}
\varphi(a(\omega \otimes (v y^\alpha + I^\alpha) \otimes (d(\mfs) x_\blam + J^\blam))) =
\sum_{v'} r^v_{v'} \, \omega v' y^\alpha  d(\mfs) x_\blam + \overline N^\blam.
\end{equation}
On the other hand,
\begin{equation} \label{equation check A wreath S n module map 9}
\begin{aligned}
a \varphi(\omega \otimes (v y^\alpha + I^\alpha) \otimes (d(\mfs) x_\blam + J^\blam))
&= a  \omega vy^\alpha d(\mfs) x_\blam + \overline N^\blam \\
&= \omega a'  vy^\alpha d(\mfs) x_\blam + \overline N^\blam \\
&=  \omega\,  (\sum_{v'} r^v_{v'}  v' y^\alpha  + z) d(\mfs) x_\blam + \overline N^\blam \\
&=      \sum_{v'} r^v_{v'} \, \omega v' y^\alpha  d(\mfs) x_\blam + \omega z d(\mfs) x_\blam + \overline N^\blam.
\end{aligned}
\end{equation}
The set $Z^\alpha$ defined above is invariant under place permutations by elements of $\sym_\alpha$, so $z' = \leftexp{d(\mfs)^*}{\hskip - .2 em z}$ is also an element of $Z^\alpha$.  
By Lemma  \ref{lemma: y v x lambda in N bar lambda},  $z' x_\blam \in \overline N^\blam$.  Therefore,
$$
\omega z d(\mfs) x_\blam = \omega d(\mfs) z' x_\blam \in \overline N^\blam.  
$$
Hence, from Equation \refc{equation check A wreath S n module map 9}, we have
\begin{equation} \label{equation check A wreath S n module map 10}
a \varphi(\omega \otimes (v y^\alpha + I^\alpha) \otimes (d(\mfs) x_\blam + J^\blam))
=   \sum_{v'} r^v_{v'} \, \omega v' y^\alpha  d(\mfs) x_\blam  +  \overline N^\blam.
\end{equation}
Comparing Equations \refc{equation check A wreath S n module map 4} 
and \refc{equation check A wreath S n module map 5}  and also Equations 
\refc{equation check A wreath S n module map 8} and 
\refc{equation check A wreath S n module map 10}, we see that $\varphi$ is both an $\symn$--module map and an $A^{\otimes n}$--module map, and therefore an $A \wr \symn$--module map.  
\end{proof}

\section{The $A$--Brauer algebras}
 
In this section, we describe the construction of the $A$--Brauer algebra, where $A$ is an algebra with involution $*$  and with a $*$--invariant trace.   The $A$--Brauer algebra can be regarded as a sort of wreath product of $A$ with the ordinary Brauer algebra, generalizing the wreath product of $A$ with the symmetric group.

The $A$--Brauer algebras   generalize the  $G$--Brauer algebras for abelian groups  $G$ defined in ~\cite{G-Brauer} as well as the  the $A$--Temperley Lieb algebras described in in ~\cite{jones-planar} or ~\cite{Goodman-Graber1}, Section 5.8.  
$G$--Brauer algebras for non--commutative groups are a particular case of the $A$--Brauer algebra construction;   they have not previously appeared in the literature.

\subsection{Algebras with involution--invariant trace}  \label{subsection Algebras with involution--invariant trace}   The definition of the $A$--Brauer algebras, in Section \ref{subsection Definition of the $A$--Brauer algebras}, requires $A$ to be an algebra with an involution $*$ and a $*$--invariant trace.  Moreover, we will show in Section \ref{section: Cellularity of A brauer algebras}
 that if $A$ is a cyclic cellular algebra with a $*$--invariant trace, then the $A$--Brauer algebras $D_n(A)$ are also cyclic cellular.   It is therefore desirable to have examples of cyclic cellular algebras with  a (natural) $*$--invariant trace. 

\begin{example}{\em Hecke algebras of type $A$.}  Let $H_n = H_{n, R}(q)$ denote the Hecke algebra of type $A$,  as in Example \ref{example: Hecke algebra is cyclic cellular}.   For any $z \in R$, there exists a unique $R$--valued trace $\tr_z$ defined on $\cup_{n\ge 1} H_n$ with the properties:
$\tr_z(1) = 1$ and $\tr_z(a T_n b) = z \tr_z(ab)$, for $a, b \in H_n$.    See ~\cite{jones-hecke}, Theorem 5.1.  Such a trace is called  Markov trace.  Note that $\tr_z\circ *$ is another trace with the same properties, and therefore, by uniqueness, $\tr_z$ is $*$--invariant.  When $z = 0$, one obtains the trace $\tr_0$ such that $\tr_0(T_\pi) = 0$ for $\pi \ne 1$ and $\tr_0(1) = 1$.  
\end{example}

\begin{example}{\em  Jones--Temperley--Lieb, Brauer, and BMW algebras.}  Let $A$ denote the $n$--strand Jones--Temperley--Lieb algebra, Brauer algebra or BMW algebra over an integral domain  $R$ with appropriate parameters, and let $\delta_0 \in R$ satisfy $e_i^2 = \delta_0 e_i$ for all $i$.   The algebra $A$ is, or can be realized as, an algebra of diagrams or tangles;  see for example Section 5 of ~\cite{Goodman-Graber1} for a description of the algebras.   The involution $*$ on $A$ acts on diagrams or tangles by flipping over a horizontal line.   $A$ has a natural $*$--invariant trace $\rm Tr$   (the Markov trace) defined by ``closing up diagrams."  The trace on the $n$--strand algebra satisfies
${\rm Tr(1)} = \delta_0^n$.  
\end{example}

\begin{example} {\em Wreath products.}   Suppose that $A$ is an $R$--algebra with involution $*$ and $*$--invariant trace $\tr$, with $\tr(1) = \delta_0$.      We can define a functional $\Tr$  on 
$A\wr \symn$ as  follows:
Let $a = a_1\otimes \cdots \otimes a_n$ be a simple tensor in $A^{\otimes n}$.   
For a $\pi \in \symn$,  let ${\sf o} = (i, \pi(i), \pi^2(i), \dots, \pi^k(i))$ be an orbit for the action of powers of $\pi$ on $\{1, 2, \dots, n\}$.   Let ${\sf o}(a) = \tr(a_{\pi^k(i)} \cdots  a_{\pi(i)}  a_i)$.  
Define  $\Tr(\pi a) =  \prod {\sf o}(a)$, where the product is over the orbits of powers of $\pi$.    One can check that $\Tr$ is a $*$--invariant trace on $A\wr \symn$ such that $\Tr(1) = \delta_0^n$.  
\end{example}

\subsection{Ordinary Brauer diagrams and Brauer algebras}
\label{subsection: Ordinary Brauer diagrams and Brauer algebras}
We review the construction of the ordinary Brauer algebras ~\cite{Brauer, Wenzl-Brauer}.
Let $\mathcal R$ denote the square $[0, 1] \times [0, 1]$.
Fix an increasing sequence $(s_i)_{i \ge 1}$ in the open interval $(0, 1)$.  Let 
$\p i$ denote $(s_i, 1) \in \mathcal R$ and let $\pbar i$ denote $(s_i, 0)$.  

An $(n,n)$--{\em Brauer diagram} (or $n$--strand Brauer diagram) is a graph with vertices $$\{\p 1, \dots \p n\} \cup \{\pbar 1, \dots, \pbar n\},$$ and with $n$ edges connecting the vertices in pairs; i.e., each vertex is adjacent to exactly one edge, and no loops are allowed. In practice, we draw the edges as curves in $\mathcal R$ which intersect the boundary of $\mathcal R$ only at their endpoints.  Two Brauer diagrams are identified if they are the same as graphs with labeled vertices, i.e. if the edges connect the same pairs of vertices. 
A  typical $4$--strand Brauer diagram is given in Figure \ref{figure: brauer diagram}.

The edges in a Brauer diagram will also be referred to as {\em strands}.  The vertices
$\p 1, \dots \p n $ in an $n$--strand Brauer diagram are called the {\em top} vertices and the
vertices $\pbar 1, \dots, \pbar n$ the {\em bottom} vertices.
\begin{figure}[ht]
$$
\inlinegraphic[scale=1]{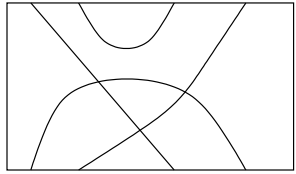}
$$
\caption{A Brauer Diagram}
\label{figure: brauer diagram}
\end{figure}

Let $R$ be a commutative ring with identity and let $\delta \in R$.  We recall the definition of the {\em $n$--strand Brauer algebra} over $R$ with parameter $\delta$, denoted $D_n = D_{n, R}(\delta)$.  
As an $R$--module, $D_n$ is  free  with basis the set of $n$--strand Brauer diagrams.  The product $xy$ of two Brauer diagrams $x$ and $y$ is a certain power of $\delta$ times a third Brauer diagram $z$, determined as follows:
stack $y$ over $x$ and compress the resulting ``tangle" vertically to fit in $\mathcal R$.  The tangle may have some number $c$ of closed loops as well as $n$ strands connecting vertices on the top and bottom edges in pairs.  Let $z$ be the Brauer diagram obtained by removing all the closed loops.  Then $xy = \delta^c z$.  This product extends to a bilinear associative product on $D_n$. 
The identity element of the algebra $D_n$ is the Brauer diagram in which $\p j$ and $\pbar j$ are connected by a strand for each $j$ ($1 \le j \le n$). 
For each Brauer diagram $d$,  let $d^*$ be the diagram obtained by flipping $d$ over a horizontal line.
Then $d \mapsto d^*$ extends to an algebra involution of $D_n$.

A {\em vertical} or {\em through} strand in an ordinary  Brauer diagram is a strand that connects a top vertex and a bottom vertex.  A {\em horizontal} 
strand is one that connects two top vertices or two bottom vertices.   

A Brauer diagram with only vertical strands is called a {\em permutation diagram}.   
The set of permutation diagrams forms a subgroup of the the group of invertible elements in $D_n$,  isomorphic to the symmetric group $\symn$.  For a permutation  $\pi \in D_n$,  $\pi^*$ is the inverse permutation.

\subsection{Order and orientation of edges} \label{subsection order and orientation of edges}

We  will need a standard ordering  and a standard orientation of the edges of a Brauer diagram.  To this end,  order the vertices of the $n$--strand Brauer diagrams by  
$$  \p 1 < \p 2 < \cdots <\p  n < \pbar 1 < \pbar 2 < \cdots < \pbar n. $$
Now we  identify  each edge in a Brauer diagram by its pair of vertices, written in increasing order.  For example, in the Brauer diagram in Figure \ref{figure: brauer diagram}
the edges are $(\p 1, \pbar 3)$,  $(\p 2, \p 3)$, $(\p 4, \pbar 2)$,  and $(\pbar 1, \pbar 4)$.  
Moreover, we  impose
a total order on the set of edges of a Brauer diagram by declaring $(v_1, v_2) < (v_1', v_2')$ if $v_1 < v_1'$.   In the example above, we have
$$
(\p 1, \pbar 3) < (\p 2, \p 3) < (\p 4, \pbar 2) < (\pbar 1, \pbar 4).
$$
The {\em standard orientation} of the edges of a Brauer diagram is from lower labeled vertex to higher labeled vertex.

\subsection{Definition of the $A$--Brauer algebras}  \label{subsection Definition of the $A$--Brauer algebras}
Let $R$ be a commutative unital algebra with distinguished element $\delta$, as in the description of the Brauer algebras.
Let $A$ be a (unital, associative, not necessarily commutative) algebra over $R$.  We suppose that $A$ has an $R$--linear algebra involution $*$  and a $*$--invariant $R$--valued trace $\tr$.   
We suppose that the trace on $A$ is normalized by $\tr(\bm 1) = \delta$,  rather than $\tr(\bm 1) = 1$.

An {\em oriented $A$--labeled Brauer diagram} is a Brauer diagram in which each strand is endowed with an orientation (not necessarily the standard orientation) and labeled by an element of $A$.  We identify two such diagrams if one is obtained from the other by reversing the orientation on some strands and simultaneously changing the $A$--valued label of these strands by applying the involution  $*$.  For example, the diagrams in Figure \ref{figure: oriented A labeled diagram}  are identified.
\begin{figure}[h!]
$$
\inlinegraphic[scale=1.2]{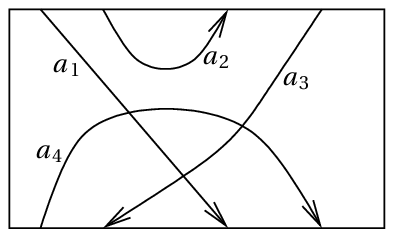} \quad = \quad \inlinegraphic[scale=1.2]{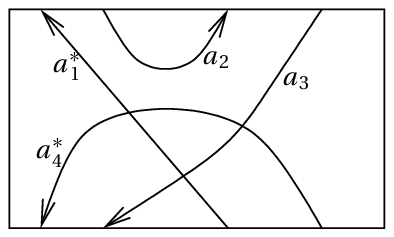}.
$$ 
\caption{Oriented $A$--labeled Brauer diagram}
\label{figure: oriented A labeled diagram}
\end{figure}
We can omit the label on strands which would be labeled by  the multiplicative identity of $A$;  we also don't have to care about the orientation of such strands because ${\bm 1}^* = {\bm 1}$ in $A$. 
We will frequently write ``$A$--Brauer diagram" as a shorthand for ``oriented $A$-labeled Brauer diagram".

As an $R$--module, the {\em $n$--strand $A$--Brauer algebra} 
$D_{n}(A) = D_{n, R, \delta}(A)$  
is  defined to be
\begin{equation}
\label{A Brauer algebra as S module}
D_n(A) = A^{\otimes n} \otimes D_n = \bigoplus_d (A^{\otimes n} \otimes d),
\end{equation}
 where the direct sum  is over  the set of $n$--strand Brauer diagrams.
  We will now explain an identification of simple tensors in $A^{\otimes n} \otimes D_n$ with  $A$--Brauer diagrams.   
We identify a simple tensor  $a_1 \otimes \cdots \otimes a_n \otimes d$,  where $d$ is a Brauer diagram, with the diagram in which the strands of $d$, taken in  increasing order, are labeled by 
$a_1, \dots, a_n$, and each strand of  $d$ is  given the standard orientation.  For example, with $d$ the $4$--strand Brauer diagram in Figure \ref{figure: brauer diagram}, we identify
$a_1 \otimes \cdots \otimes a_4 \otimes d$ with the $A$--Brauer diagram shown in Figure \ref{figure: oriented A labeled diagram}.

\begin{lemma} \label{direct sum principle for A Brauer algebras}
 Let  $P$ be any subset of the set of ordinary $n$--strand Brauer diagrams, and let 
 $\{P_1, \dots, P_k\}$ be any set partition of $P$.  Let $M$ be the span of the $A$--Brauer diagrams whose underlying Brauer diagram lies in $P$,  and for
each $i$, let $M_i$ be the span of the set of $A$--Brauer diagrams whose underlying ordinary Brauer diagram lies in $P_i$.   Then $M = \bigoplus_i M_i$ as $R$--modules.
\end{lemma}

\begin{proof} Using the identification of $A$--Brauer diagrams with simple tensors, we have
$M = \bigoplus_{d \in P}  (A^{\otimes n} \otimes d)$ and 
$M_i = \bigoplus_{d \in P_i}  (A^{\otimes n} \otimes d)$.   Since $\{P_i\}$ is a set partition of $P$,  the result follows.
\end{proof}

Next we describe the product of two $A$--Brauer diagrams; the product is  another $A$--Brauer diagram.  Given two $A$--Brauer diagrams $X$ and $Y$ with $n$ strands, to form the product
$XY$,  first stack $Y$ over $X$ as for ordinary Brauer diagrams.   The resulting ``tangle" may contain some closed strands as well as $n$  non--closed strands connecting two vertices;   each strand is composed of one or more strands from $X$ and $Y$.  Give each composite strand $s$  an orientation, arbitrarily;   for non--closed strands we can take the standard orientation, but for closed strands there is no preferred orientation.   Make the orientations of the component strands of $s$ from $X$ and $Y$  agree with the chosen orientation of $s$,  changing labels by applying $*$ whenever the orientation of a component strand is reversed.  Label each strand by the product of the labels of its component strands, written from right to left in the order in which the labels are encountered as the strand is traversed according to its chosen orientation.   For closed strands, this product is defined only up to cyclic permutation of the factors, since there is no preferred starting point for traversing the strand.

 Finally, remove any $A$--labeled closed strands, but for each such strand multiply the diagram by the $R$--valued trace of its label;  since the trace of a product  is invariant under cyclic permutation of  the factors, this factor in $R$ is well--defined.   Moreover, the element of  $R$ obtained is independent of the choice of orientation of the closed strand; 
 if the orientation of a closed strand is reversed, its label is replaced by $*$ applied to the old label, and, since the trace is $*$--invariant, again the factor in $R$ is unchanged.  
 
 So far the product of two $A$--Brauer diagrams has been described as a multiple of a third, $XY = \alpha Z$,  where $\alpha \in R$,  but using the identification of 
 $A$--Brauer diagrams with simple tensors, the factor $\alpha$ can be absorbed into the $A$--valued label of any strand of $Z$.

 \ignore{The  $A$--Brauer diagrams displayed in Figure \ref {some A labeled Brauer diagrams} will be employed in several examples below.}
 \begin{figure}[h!]
\begin{equation*}
 X = \inlinegraphic[scale=1.1]{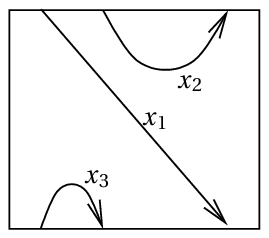}, \quad  
 X' = \inlinegraphic[scale=1.1]{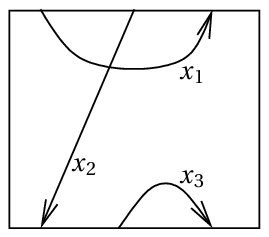},  \quad
 Y = \inlinegraphic[scale=1.1]{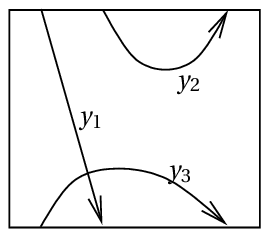}.
\end{equation*}
\caption{Some $A$--Brauer diagrams}
 \label{some A labeled Brauer diagrams}
\end{figure}

 Here are two examples of the product of $A$--Brauer diagrams.  Let $X$, $X'$, and $Y$ be as 
 in Figure \ref{some A labeled Brauer diagrams}.  Then 
 $$
 XY =   \inlinegraphic[scale=1.1]{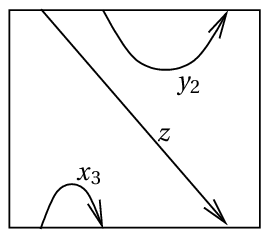}, \text{ with  $z = x_1 y_3^* x_2 y_1$,}
 $$
and
 $$
 X'Y =  \alpha  \inlinegraphic[scale=1.1]{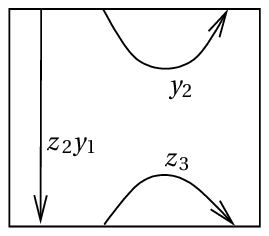} = \inlinegraphic[scale=1.1]{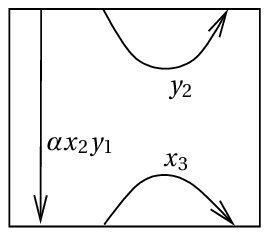},
 \text{ where $\alpha = \tr(y_3^* x_1) = \tr(x_1^*  y_3)$.}
 $$

 Fix two Brauer diagrams $x$ and $y$, and let $z$ be the Brauer diagram obtained by stacking
 $y$ over $x$ and removing closed loops.    The procedure described above, applied to $A$--Brauer diagrams $X$ and $Y$ with underlying Brauer diagrams $x$ and $y$, defines a multilinear map  $A^{2n} \to  A^{\otimes n} \otimes z$, hence a bilinear map
 $$ (A^{\otimes n} \otimes x) \times ( A^{\otimes n} \otimes y) \to  A^{\otimes n} \otimes z.$$
 Taking Equation (\ref{A Brauer algebra as S module}) into account, we see that the product extends to a bilinear product $D_n(A) \times D_n(A) \to  D_n(A)$.  One can check that this product is associative.

 For each oriented $A$--labeled  Brauer diagram $X$,  let $X^*$  be the diagram obtained by flipping
 $X$ over a horizontal line.  For example,  with 
 $X$ as in Figure \ref{some A labeled Brauer diagrams},
  we have
 $$
 X^* = \inlinegraphic[scale=1.1]{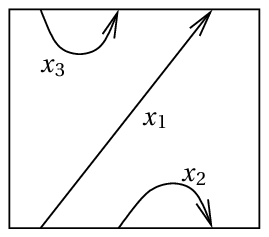} = \inlinegraphic[scale=1.1]{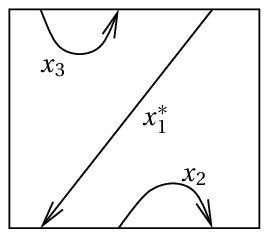}.
 $$
 
 \begin{lemma} $X \mapsto X^*$  extends to an $R$--linear algebra involution on $D_n(A)$.   
 \end{lemma}
 
 \begin{proof}   Straightforward.
 \end{proof}
 
  \begin{remark}\label{embedding of ordinary Brauer algebra}  It follows from Equation \refc{A Brauer algebra as S module},   the rule for multiplying $A$--Brauer diagrams, and the definition of the involution $*$,  that the ordinary $n$--strand Brauer algebra $D_n$ imbeds as a unital $*$--subalgebra  of $D_n(A)$,  spanned by  unlabeled Brauer diagrams.
 In particular the symmetric group $\symn$ can be identified with the set of unlabeled permutation diagrams in $D_n(A)$.  $D_n(A)$ is thus an $\symn$--bimodule, with the left and right actions of the symmetric group given by left and right multiplication by permutation diagrams.
 \end{remark}

 \begin{remark}  $D_1(A) \cong A$ as algebras with involution and trace.    It is convenient to define $D_0(A)$ to be $R$ with the trivial involution.
 \end{remark}
 
Let $e_j$ and $s_j$ denote the $(n, n)$--Brauer diagrams:
$$
e_j =  \inlinegraphic[scale=.7]{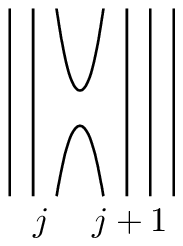}\qquad
s_j =  \inlinegraphic[scale= .7]{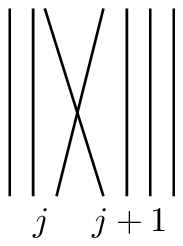} 
$$
 For $a \in A$ and $1 \le i \le n$, let $a^{(i)}$  be the  $A$--Brauer diagram whose underlying Brauer diagram is the identity diagram, with strands oriented from top to bottom, with the $i$--th strand labeled by $a$.   
 
 \begin{lemma} \label{lemma: generators of D n of A}
 $D_n(A)$ is generated as a unital algebra by the ordinary Brauer diagrams
 $e_i$ and $s_i$  for $1 \le i \le n-1$, and by the   diagrams $a^{(1)}$  for $a \in A$.  
 One has $e_i^* = e_i$,  $s_i^* = s_i$ and $(a^{(1)})^* = (a^*)^{(1)}$.  
 \end{lemma}
 
  \begin{proof}   Straightforward.
 \end{proof}

 \begin{remark} \label{remark basis of A Brauer algebra}
  If $A$ is free as an $R$--module, then $D_n(A) = A^{\otimes n} \otimes D_n$ is also free as an $R$--module, with basis consisting of $n$--strand Brauer diagrams whose strands are given the standard orientation and labeled by basis elements of $A$. 
 \end{remark}
 
\subsection{Filtration by rank, and the embedding of the wreath product $A \wr \symn$}   The {\em rank} of an ordinary or $A
$--labeled Brauer diagram is the number of through strands.  
Note that the rank of an $n$--Brauer diagram has the same parity as $n$.
It follows immediately from the definition of 
the multiplication that for two $A$--Brauer diagrams $X$ and $Y$,  $\rank(XY) \le \min\{\rank (X), \rank (Y)\}
$.  Hence if we write $W_r^n$ for the span of $A$--Brauer diagrams with rank no more than $r$ in $D_n(A)$,  then 
$W_r^n$ is an ideal and
 $$
 (0) \subset W_0^n \subset W_1^n \subset \cdots \subset W_n^n = D_n(A).
 $$
 Write $V_s^n$ for the span of $A$--Brauer diagrams with rank exactly $s$ in $D_n(A)$.  
 
 \begin{lemma} \label{direct sum by rank}
$D_n(A) = \bigoplus_s V_s^n$,  and $W_r^n = \bigoplus_{s \le r} V_s^n$, as $R$--modules.
 \end{lemma}
 
 \begin{proof}  This follows from Lemma \ref{direct sum principle for A Brauer algebras}.
 \end{proof}
 
 \begin{lemma}  \label{identification of V n n with wreath product algebra}
 $V_n^n$ is a $*$--subalgebra of $D_n(A)$ isomorphic to the wreath product algebra $A \wr \symn$, as algebras with involution.
 \end{lemma}
 
 \begin{proof}  $V_n^n$ is closed under multiplication and under the involution $*$.  Moreover,  $V_n^n$ is isomorphic to $A \wr \symn$ as  $R$--modules, since $V_n^n = \bigoplus_\pi  A^{\otimes n} \otimes \pi$,  where the direct sum is over permutation diagrams.  It is straightforward to check that the isomorphism of $R$--modules $V_n^n \cong A \wr \symn$ also respects multiplication and involution.
 \end{proof}
 
 \subsection{Examples: the $G$--Brauer algebras}  Throughout this section, let $R$ be a commutative ring with distinguished element $\delta$ as above.
 
 Let $G$ be an arbitrary group. There is a canonical algebra involution $*$  on the group algebra
$A = RG$, namely the linear extension of the inverse map $g \mapsto g\inv$ on $G$.   The $R$--valued trace functions on $A = RG$ correspond to class functions on $G$;  a trace is $*$--invariant when $\tr(g\inv) = \tr(g)$ for all $g \in G$.  According to our convention, we have to normalize such a class function by $\tr(\bm 1) = \delta$.  For any such choice of a trace function, we can form the $A$--Brauer algebras with $A = RG$.   Since $A$ is free as an $R$--module with basis $G$,
by Remark \ref{remark basis of A Brauer algebra},  the $A$--Brauer algebra  $D_n(A)$  is free as an $R$--module with basis the set of $G$--labeled oriented $n$--strand Brauer diagrams (with two such diagrams identified if one is obtained from the other by reversing the orientation of some strands and simultaneously inverting the $G$--valued label of the strand).
We call these algebras $G$--Brauer algebras.

When $G$ is an abelian group, there is another choice for the involution on 
$A = RG$,  namely the identity map.  With this choice of involution, any class function on $G$ yields an (involution--invariant) trace on $A = RG$; again, we have to normalize the trace by $\tr(\bm 1) = \delta$.  For any choice of the trace, we can form the $A$--Brauer algebras.  Since the involution on $A$ is trivial, the orientation of the strands of $A$--labeled Brauer diagrams becomes irrelevant.  Thus the $A$--Brauer algebra $D_n(A)$ is free as an $R$--module with basis the set of {\em unoriented} $G$--labeled $n$--strand Brauer diagrams.
These $G$--Brauer algebras are the same as those defined by  Parvathi and Savithri  ~\cite{G-Brauer}.

Thus when $G$ is abelian, we have two slightly different constructions of $G$--Brauer algebras.  The Parvathi--Savithri construction is the more interesting one for this paper, because $A = RG$ is a cyclic cellular algebra with the trivial involution, assuming $R$ has sufficiently many roots of unity, while  $A = RG$ cannot be cyclic cellular with any non--trivial involution, by Proposition \ref{proposition characterization of abelian s cellular algebras}. 

Note that for any $G$,  abelian or not, there is a canonical involution--invariant trace  on $A = RG$ given by
$\tr(g) = 0$ if $g \ne \bm 1$ and $\tr(\bm 1) = \delta$.

 \subsection{Inclusion, closure, and trace}
 For $n \ge 1$,  and for $X$ an $A$--Brauer diagram with $n$ strands,  let $\iota(X)$ be 
 the diagram obtained by adding an unlabeled strand  to $X$, connecting $\p {n+1}$ to
$\pbar {n+1}$.
 $$
\iota: \quad \inlinegraphic{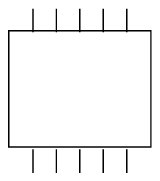} \quad \mapsto \quad 
\inlinegraphic{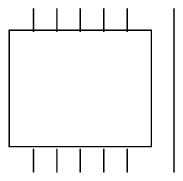}
$$
Then $\iota$ determines a unital  algebra homomorphism $\iota : D_n(A) \to D_{n+1}(A)$. 
 
 For $n \ge 1$ define a map     $\cl$  from  oriented  $A$--labeled Brauer diagrams with $n$ strands into $D_{n-1}(A)$ as follows.  First ``partially close" a given $A$--Brauer diagram $X$ by adding an additional smooth curve connecting $\p n$ to $\pbar n$,
$$
 \inlinegraphic{tangle_box2} \quad \mapsto \quad 
\inlinegraphic{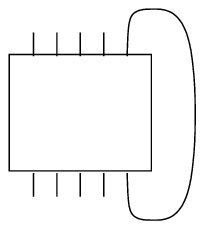}.
$$
The resulting ``tangle" contains a closed curve  precisely when the original diagram already had a strand connecting $\p n$ to $\pbar n$;   we can suppose  that this strand is oriented from $\p n$ to $\pbar n$ and labeled by $a \in A$;  in this case, 
remove this loop and replace it with a factor of $\tr(a)$.    Otherwise, the added curve is part of a composite strand containing two strands from $X$;  compute the $A$--valued label of this composite strand as in the rule for multiplying diagrams.   The map $\cl$  on oriented  $A$--labeled diagrams determines an $R$--linear map
$\cl : D_n(A) \to D_{n-1}(A)$.  Moreover, $\cl$ is a (non--unital)  $D_{n-1}(A)$---$D_{n-1}(A)$  bimodule map, and $\cl \circ * = *\circ \cl$ for all $n$.     Note that $\cl: D_1(A) \to D_0(A) = R$ is just the trace on $D_1(A) \cong A$. 
For $X$ an $A$--Brauer diagram with $n$ strands,  we have
$$
\cl( \iota(X) e_n) = X;
$$
it follows that $\iota$ is injective, so we can regard $D_n(A)$ as a unital subalgebra of $D_{n+1}(A)$.   
 
 We define $\Tr: D_n(A) \to R$ by $\Tr = \cl\circ  \cdots \cl \circ \cl$.    Then, as for other diagram or tangle algebras, $\Tr$ is a trace.   It follows from the properties of $\cl$ discussed above that
 $\Tr$ is $*$--invariant.

\subsection{The $A$--Brauer category}  We can imbed the $n$--strand $A$--Brauer algebras for all $n$ in a category as follows.  The category $\mathcal B_A$  has as objects the integers $0, 1, 2, \dots$.    

To describe the morphisms in the category, first we describe $(k, \ell)$--Brauer diagrams, where $k$ and $\ell$ have the same parity.  These are are ``Brauer diagrams" with $k$ upper vertices $\p1, \dots, \p k$ and $\ell$ lower vertices $\pbar 1, \dots, \pbar \ell$, with the vertices connected in pairs by edges.  We order the vertices of $(k,\ell)$--Brauer diagrams by
$$  \p 1 < \p 2 < \cdots <\p  k < \pbar 1 < \pbar 2 < \cdots < \pbar \ell,$$
and we use this ordering to define a total ordering of the strands of $(k,\ell)$--Brauer diagrams, and a standard orientation of strands, as in Section \ref{subsection order and orientation of edges}.

Let $D_{k, \ell}$ be the free $R$--module with basis the set of $(k,\ell)$--Brauer diagrams.

When $k$ and $\ell$ do not have the same parity, we take $\Hom(k, \ell) = (0)$. When  $k$ and $\ell$   have the same parity we set 
\begin{equation} \label{hom spaces direct sum as R modules}
D_{k, \ell}(A) = \Hom(k, \ell) = A^{\otimes (k+\ell)/2} \otimes D_{k,\ell} = \bigoplus_d A^{\otimes (k+\ell)/2} \otimes d,
\end{equation}
where the direct sum runs over all  $(k, \ell)$--Brauer diagrams.
  We identify simple tensors in $\Hom(k, \ell) = A^{\otimes (k+\ell)/2} \otimes D_{k,\ell}$ with oriented $A$--labeled $(k, \ell)$--Brauer diagrams.  This is done in the same way as was described previously in the case $k = \ell$, in the description of the $k$--strand $A$--Brauer algebra.

  \begin{remark} \label{remark basis of hom spaces}
  As in Remark \ref{remark basis of A Brauer algebra}, if $A$ is free as an $R$--module, then each Hom space $D_{k, \ell}(A) = \Hom(k, \ell)$ is also a free $R$--module with basis consisting of $(k, \ell)$--Brauer diagrams with strands endowed with the standard orientation and labeled by basis elements of $A$.  
  \end{remark}

If $X \in \Hom(\ell, m)$ and $Y \in \Hom(k, \ell)$ are $A$--Brauer diagrams, then we can form the product $X Y$,  which is an oriented $A$--labeled $(k, m)$--Brauer diagram.  The definition of the product is essentially the same as the definition of the product in $D_n(A)$ given previously.  The product extends to a bilinear product
 $\Hom(\ell, m) \times  \Hom(k, \ell) \to  \Hom(k, m)$.  Moreover, the product is associative; i.e., if $Z$ is an $A$--Brauer diagram in $\Hom(j, k)$, then
 $(XY) Z = X(YZ)$ in $\Hom(j, m)$.   Thus $\mathcal B_A$ is a category whose Hom spaces are $R$--modules, with bilinear composition of morphisms.  Note that the algebra $\End(n)$ is isomorphic to $D_n(A)$. 
 
 For an oriented $A$--labeled $(k, \ell)$--Brauer diagram $X$, let $X^*$ be the oriented $A$--labeled $(\ell, k)$--Brauer diagram obtained by flipping $X$ over a horizontal axis.  Then $*$ determines a linear map from $\Hom(k, \ell)$ to $\Hom(\ell, k)$  for each $k$ and $\ell$; moreover, $*$ is a contravariant functor from $\mathcal B_A$ to itself; that is, $(XY)^* = Y^*X^*$, whenever the product $XY$ is defined.
 
 The category $\mathcal B_A$ has a tensor or monoidal structure, described as follows. The tensor product of objects is $n \odot m = n + m$.   Given oriented $A$-labeled Brauer diagrams $X \in \Hom(k, \ell)$ and $Y \in \Hom(m, n)$, their tensor product
 $X \odot Y \in \Hom(k + m, \ell + n)$ is given by horizontal juxtaposition of diagrams.
Fix ordinary Brauer diagrams $x \in D_{k, \ell}$ and $y \in D_{m,n}$ and let $z$ be the
$(k + m, \ell + n)$--Brauer diagram obtained by horizontal juxtaposition of $x$ and $y$.  Then $(X, Y) \mapsto X \odot Y$,  
applied to $A$--Brauer diagrams $X$ and $Y$ with underlying Brauer diagrams $x$ and $y$, defines a multilinear map from 
$$A^{(k + \ell + m + n)/2} \to
A^{\otimes (k + \ell + m + n)/2} \otimes z,
$$
 hence a bilinear map 
$$
(A^{\otimes (k  + \ell)/2}  \otimes x)  \times (A^{\otimes (m + n )/2}  \otimes y) \to 
A^{\otimes (k + \ell + m + n)/2} \otimes z.
$$
 Taking Equation (\ref{hom spaces direct sum as R modules}) into account, we see that the product extends to a bilinear product $\Hom(k, \ell)  \times \Hom(m, n)  \to  \Hom(k + m, \ell + n)$.   One can check that $\mathcal B_A$ satisfies the axioms of a monoidal tensor category (see for example, ~\cite{bakalov-kirillov}, page 12).  We will not need this, so we will not go into more detail here.
 
The bilinear tensor product operation 
 $$\odot: \Hom(k, \ell)  \times \Hom(m, n)  \to  \Hom(k + m, \ell + n)
 $$
  determines a linear map
$$
\odot : \Hom(k, \ell)  \otimes \Hom(m, n)  \to  \Hom(k + m, \ell + n).
$$  
In particular, identifying $D_n(A)$ with $\End(n)$, we have a linear map
$$
\odot: D_n(A) \otimes  D_m(A) \to  D_{n + m}(A).
$$

 \begin{lemma} \label{lemma:  injectivity of horizontal juxtaposition}
 The linear map
  $
\odot: D_n(A) \otimes  D_m(A) \to  D_{n + m}(A)
$ 
is injective.
 \end{lemma}
 
 \begin{proof}
 Note this is obvious if $A$ is free as an $R$--module, which is the only case we will actually need.  In this, case, the Hom spaces are also free, as described in Remark \ref{remark basis of hom spaces}, and the map
takes basis elements to basis elements. 

For the general case, fix an ordinary  $m$--Brauer diagram $x$ and an ordinary
$n$--Brauer diagram $y$, and let $z$ be the ordinary $m+n$--Brauer diagram obtained by horizontal juxtaposition of $x$ and $y$.  Then because of 
Equation \refc{A Brauer algebra as S module}, we only have to check injectivity of the
restriction
$
\odot: (A^{\otimes m} \otimes x) \otimes  (A^{\otimes n} \otimes y) 
\to A^{\otimes(m+n)} \otimes z.
$
But this map can be identified with a linear map $A^{\otimes m} \otimes A^{\otimes n} \to A^{\otimes(m+n)}$,  and this linear map is simply a permutation of tensor places.   The following example will clarify this.
 \end{proof}
 
 \begin{example}  Consider the $A$--Brauer diagrams $X$ and $Y$ from Figure \ref{some A labeled Brauer diagrams}.  Let $x$ and $y$ be the underlying Brauer diagrams of $X$ and $Y$, and let $z$ be the underlying Brauer diagram of $X \odot Y$.   Then under the identification of oriented $A$--labeled diagrams with simple tensors, $X$ is identified with $(x_1 \otimes x_2 \otimes x_3) \otimes x$,  $Y$ with $(y_1 \otimes y_2 \otimes y_3) \otimes y$, and 
 $X \odot Y$ with $(x_1 \otimes x_2 \otimes y_1 \otimes y_2 \otimes x_3 \otimes y_3) \otimes z$.  
 \end{example}
 
 Recall that $V_0^{2f}$ denotes the span of $A$--Brauer diagrams of rank zero in $D_{2f}(A)$.   Any $A$--Brauer diagram $Z$ of rank zero evidently has a factorization $Z = X Y^*$, where
 $X$ and $Y$  are  oriented $A$--labeled $(0, 2f)$--Brauer diagrams.  $X$ and $Y^*$  are simply the
 bottom and top halves of $Z$.

 \begin{lemma} \label{tensor decomposition of rank zero diagrams}
 $X \otimes Y \mapsto X Y^*$ determines an isomorphism
 $$
 \Hom(0, 2f) \otimes \Hom(0, 2f) \to  V_0^{2f}.
 $$
 \end{lemma}
 
 \begin{proof} Surjectivity follows from the remark preceding the statement of the lemma, and the proof of injectivity is similar to the proof of Lemma \ref{lemma:  injectivity of horizontal juxtaposition}. 
 \end{proof}
 
 \subsection{Bases of $A$--Brauer algebras}
 
 \begin{definition}  Let $0 \le k \le n$.  A $(k, n-k)$--shuffle is a permutation $\pi \in \symn$ such that
 $\pi(i) < \pi(j)$  whenever $1 \le i < j \le k$  or $k+1 \le i < j \le n$.
 \end{definition}
 
 \begin{remark}  If $0 < k < n-k$,  the  $(k, n-k)$ shuffles are just the distinguished left coset representatives of 
 $\sym_k \times \sym_{n-k}$ in $\symn$.  If $k = 0$ or $k = n$ the identity permutation is the unique $(k, n-k)$ shuffle.
 \end{remark}
 
 \begin{lemma} \label{lemma factorization of Brauer diagrams}
  Let $X$ be an $n$--strand $A$--Brauer diagram  of rank $s = n - 2f$.  Then $X$ has a unique factorization
 $$
 X = \alpha(X_0 \odot X_1) \beta^*,
 $$
 where $\alpha$ and $\beta$ are $(2f, s)$--shuffles, $X_0$ is a $2f$--strand $A$--Brauer diagram  of rank $0$,  and $X_1$ is an $s$--strand $A$--labeled  permutation diagram .
 \end{lemma}
 
\begin{proof}  Let 
$
\p i_1 < \cdots < \p i_{2f}
$
be the top vertices adjacent to horizontal strands of $X$ and
$
\p j_1 < \cdots < \p j_{s}
$
the top vertices adjacent to vertical strands of $X$. Likewise, let
$
\pbar k_1 < \cdots < \pbar k_{2f}
$
be the bottom vertices adjacent to horizontal strands of $X$ and
$
\pbar l_1 < \cdots < \pbar l_{s}
$
the bottom vertices adjacent to vertical strands of $X$.
Set
$$
\alpha = \begin{pmatrix}
1  & \dots  &2f &2f+1 & \dots &n \\
 k_1 &\dots & k_{2f} & l_1 &\dots &  l_s\\
\end{pmatrix}
$$
and
$$
\beta = \begin{pmatrix}
1  & \dots  &2f &2f+1 & \dots &n \\
 i_1 &\dots & i_{2f} & j_1 &\dots &  j_s\\
\end{pmatrix}
$$
Consider $X' = \alpha^*X \beta$.  Then $X'$ has horizontal strands connecting $\p 1, \dots \p {2f}$ 
and $\pbar 1, \dots \pbar {2f}$  in pairs and  vertical strands connecting
$\p {2f+1}, \dots, \p n$ and $\pbar {2f+1}, \dots, \pbar n$ in pairs.  Hence $X' = X_0 \odot X_1$, where
$X_0$ and $X_1$ are as in the statement of the lemma.  Thus
$$
X = \alpha X' \beta^* =  \alpha(X_0 \odot X_1) \beta^*,
$$
as required.

The factorization is unique because specifying the  shuffles $\alpha$ and $\beta$  is equivalent to 
specifying the set of vertices adjacent to horizontal and to vertical strands, and $X$ together with 
$\alpha$ and $\beta$ determine $X_0$ and $X_1$.
\end{proof} 

\begin{remark}  If $d$ is an ordinary $n$--strand Brauer diagram of rank $s$, then $d$ has a unique
factorization
$$
d = \alpha (d_0 \odot d_1) \beta^*,
$$
where $\alpha$ and $\beta$ are $(2f, s)$ shuffles, $d_0$ is an ordinary $2f$--strand Brauer diagram of rank $0$, and $d_1$ is an ordinary $s$--strand permutation diagram.   Moreover, if $X$ is an $A$--Brauer diagram with underlying ordinary Brauer diagram $d$, and
$X = \alpha(X_0 \odot X_1) \beta^*$ is the factorization of $X$, then
$d =  \alpha(d_0 \odot d_1) \beta^*$,  where $d_0$ and $d_1$ are the underlying ordinary Brauer diagrams of
$X_0$ and $X_1$. 
\end{remark}

Fix $s \le n$ with $n -s$ even and let $f = (n-s)/2$.  Fix $(2f, s)$--shuffles $\alpha$ and $\beta$.  Let 
$P_{\alpha, \beta}$ be the set of ordinary Brauer diagrams $d$  of rank $s$ with factorization
$
d = \alpha (d_0 \odot d_1) \beta^*
$.  The sets $P_{\alpha, \beta}$ for different pairs $(\alpha, \beta)$ are mutually disjoint.  Let 
$$
V_{\alpha, \beta} = \alpha(V^{2f}_0 \odot V^s_s)\beta^*.
$$
Then $V_{\alpha, \beta}$ is the span of those $A$--Brauer diagrams of rank $s$ whose underlying ordinary Brauer diagram lies in $P_{\alpha, \beta}$.  

\begin{lemma} \label{direct sum decomposition of rank s subspace}
$ V^n_s = \bigoplus_{\alpha, \beta} V_{\alpha, \beta}.$
\end{lemma}

\begin{proof}  This follows from the discussion above and Lemma \ref{direct sum principle for A Brauer algebras}. 
\end{proof}

For the rest of this section, assume that $A$ is a free $R$--module, so that each $D_{k, \ell}(A)$ is also a free $R$--module, by Remark \ref{remark basis of hom spaces}. 

For each $s$ with $s \le n$ and $n-s$ even, let $f = (n-s)/2$.  Let $\mathbb B_s$ be any $R$--basis of 
$V^s_s$.    Let $\mathbb A_f$  be the basis of $D_{0, 2f}(A)$ consisting of 
$(0, 2f)$--Brauer diagrams endowed with the standard orientation and labeled by basis elements of $A$.
Define
$$
\mathcal B^n_s = \{\alpha(X Y^* \odot b) \beta^* : \alpha \text{ and } \beta \text{ are } (2f, s)\text{--shuffles}, X, Y \in \mathbb A_f, \text{ and } b \in \mathbb B_s\}.
$$

\begin{lemma} \label{lemma basis of the A Brauer algebra}
$\mathcal B^n_s$ is an $R$--basis of $V^n_s$ and $\cup_s \mathcal B^n_s$ is an $R$--basis of $D_n(A)$.  
\end{lemma}

\begin{proof}  By Lemmas \ref{lemma:  injectivity of horizontal juxtaposition} and \ref{tensor decomposition of rank zero diagrams},  
$$
 \{X Y^* \odot b :  X, Y \in \mathbb A_f, \text{ and } b \in \mathbb B_s\}
$$
is an $R$--basis of $V^{2f}_0 \odot V^s_s$.    Since, for fixed $\alpha$ and $\beta$, 
$Z \mapsto \alpha Z \beta^*$ is a linear isomorphism from $V^{2f}_0 \odot V^s_s$ to $V_{\alpha, \beta}$,
it follows that
$$
\{\alpha(X Y^* \odot b) \beta^* :  X, Y \in \mathbb A_f, \text{ and } b \in \mathbb B_s\}
$$
is an $R$--basis of $V_{\alpha, \beta}$.   Thus, by Lemmas  \ref{direct sum by rank} and  \ref{direct sum decomposition of rank s subspace}, it follows that $\mathcal B^n_s$ is a basis of $V^n_s$ and that
$\cup_s \mathcal B^n_s$ is a basis of $D_n(A)$.  
\end{proof}

\section{Cellularity of  $A$--Brauer algebras} \label{section: Cellularity of A brauer algebras}

 In this section, we will prove the following theorem:

\begin{theorem}  \label{theorem cellularity of A Brauer algebras}
Let $A$ be an algebra with an algebra involution $*$ and a
$*$--invariant trace $\tr$.  Suppose that
$A$ is  cyclic cellular.  Then for all $n \ge 1$,  the $A$--Brauer algebra $D_n(A)$ 
is a cyclic cellular  algebra. 
\end{theorem}

The proof of this theorem is similar to proofs in the literature for the cellularity of Brauer or BMW algebras  
~\cite{Enyang1, Xi-BMW}.   We will base our computations on some arguments from a well known, but unpublished paper ~\cite{Morton-Wassermann} of Morton and Wassermann on bases of the BMW algebras. 
A similar but necessarily more complicated proof of the cellularity of cyclotomic BMW algebras was given in  ~\cite{goodman-2008}. It would be possible here  to use  the method of iterated inflations of Koenig and Xi ~\cite{KX-Morita},  but we prefer not to take this approach, as we can exhibit explicit cellular bases of the $A$--Brauer  algebras, based on our cellular basis of the wreath product algebras  $A \wr \sym_s$.

Assume that $A$ is a cyclic cellular algebra over an integral domain $R$, and that $A$ has   an involution--invariant  $R$--valued trace.  
We adopt the notation of Section \ref{Cellularity of wreath products} regarding the cellular structure of $A$;  in particular, the partially ordered set in the cell datum of $A$ is denoted by $(\Gamma, \ge)$.  By Theorem \ref{theorem cellularity of wreath products}, for all $s \ge 1$, 
$A\wr \sym_s$ is cyclic cellular.  The partially ordered set in the cell datum for $A\wr \sym_s$ is
$(\Lambda^\Gamma_s, \GammaDominates)$.  The  cellular basis of $A\wr \sym_s$  is
$$
\mathcal B = \{\m \blam \mfs v \mft w :  \blam \in \multipart, \text{ and }  (\mfs, v), (\mft, w) \in \mathcal T(\blam)\},
$$
where $\mathcal T(\blam)$ denotes the set of pairs $(\mfs, v)$ with $\mfs$ a standard $\blam$--tableau and $v \in V^{\alpha(\blam)}$.     For fixed $s$ and $\blam$,   define the ideals $N^\blam$ and $\overline N^\blam$ 
in $A \wr \sym_s$, as in Definition \ref{definition of standard ideals in wreath product algebra}. 

Let $n \ge 1$.  We will now propose a cell datum for $D_n(A)$.  
Define
$$\Lambda = \{(s, \blam) :  s \le n \text{ and } n-s \text{ even},  \text{ and } \blam \in \Lambda^\Gamma_s\},
$$  
with partial order $\GammaDominates$ defined by $(s, \blam) \GammaDominates (s', \blam')$ if $s < s'$ or if $s = s'$ and 
 $\blam \GammaDominates \blam'$.    For $(s, \blam) \in \Lambda$,  let $f = (n-s)/2$.
Let $\mathbb A_f$  be the basis of $D_{0, 2f}(A)$ consisting of $(0, 2f)$--Brauer diagrams, endowed with the standard orientation, with strands labeled by elements of the cellular basis of $A$. 
Define $\mathcal T(s, \blam)$ to be the set of all
$(\alpha, X, \mfs, v)$ where $\alpha$ is a $(2f, s)$--shuffle, $X \in \mathbb A_f$,  and $(\mfs, v) \in \mathcal T(\blam)$.    For  $(\alpha, X, \mfs, v)$ and $(\beta, Y, \mft, w)$ in $\mathcal T(s, \blam)$, set

$$
\mm {(s, \blam)}  {(\alpha, X, \mfs, v)} {(\beta, Y, \mft, w)}
= \alpha(XY^* \odot \m \blam \mfs v \mft w) \beta^*,
$$
where we have identified $V^s_s$ with $A \wr \sym_s$, using Lemma \ref{identification of V n n with wreath product algebra}.   Let $\mathcal D$  denote the set of all
$\mm {(s, \blam)}  {(\alpha, X, \mfs, v)} {(\beta, Y, \mft, w)}$,   as $(s, \blam)$  varies in $\Lambda$, and
 $(\alpha, X, \mfs, v)$ and $(\beta, Y, \mft, w)$ vary in  $\mathcal T(s, \blam)$.  Our claim is that 
 $\mathcal D$ is a cellular basis of $D_n(A)$.

\begin{lemma} \label{lemma:  basis of A Brauer 2}
$\mathcal D$ is an $R$--basis of $D_n(A)$.
\end{lemma}

\begin{proof} This follows immediately from Lemma \ref{lemma basis of the A Brauer algebra}. 
\end{proof}

For fixed $(s, \blam) \in \Lambda$, define $M^{(s, \blam)}$ to be the span of basis elements
$\mm {(s', \, \bmu)}  {(\alpha, X, \mfs, v)} {(\beta, Y, \mft, w)}$ with $(s', \bmu) \GammaDominates (s, \blam)$,  and $\overline M^{(s, \blam)}$ to be the span of such basis elements with
 $(s', \bmu) \StrictlyGammaDominates (s, \blam)$.  We will use the following observation several times.
 
 \begin{lemma}  \label{lemma: subset of M bar s lambda}
 For $(s, \blam) \in \Lambda$ and $f = (n-s)/2$, we have
 $$
 V^{2f}_0 \odot \overline N^\blam \subseteq \overline M^{(s, \blam)}.
 $$
 \end{lemma}
 
 \begin{proof}  Evident.
 \end{proof}

Fix $s \le n$ with $n-s$ even and let $f = (n-s)/2)$.  
Let $X_0$ be the ordinary $(0, 2f)$--Brauer diagram with strands $(\pbar 1, \pbar 2)$, 
$(\pbar 3, \pbar 4)$, \dots, $(\pbar {2f-1}, \pbar {2f})$.   Note that 
$X_0 X_0^* = e_1 e_3 \cdots e_{2f-1}$.  
For $X \in \mathbb A_f$,  let 
$$(\pbar i_1, \pbar j_1) < (\pbar i_2, \pbar j_2) < \cdots < (\pbar i_f, \pbar j_f)
$$
be the edges of $X$ and let $a_1, \dots,  a_f$ be the elements of the cellular basis of $A$ labeling these edges.  Let $\pi(X) \in \sym_{2f}$  denote the permutation
$$
\pi(X)= \begin{pmatrix}
1 & 2 & 3 & 4 &\dots & 2f-1 & 2f \\
i_1 & j_1 & i_2 & j_2 & \dots & i_f & j_f \\
\end{pmatrix},
$$
and let $a(X) = \prod_{j = 1}^f  a_j \power {2j} \in A^{\otimes 2f} \subset A \wr \sym_{2f} \subset D_{2f}(A)$.  

\begin{lemma} \label{lemma:  factorization of V 2f 0}
 Let $X \in \mathbb A_f$.  Then
$
X = \pi(X) a(X) X_0
$.
\end{lemma}

\begin{proof}  It is straightforward to check that the oriented $A$--labelled  $(0, 2f)$--Brauer diagram given in factored form on the right hand side has the same oriented labeled edges as does $X$. 
\end{proof}

Set 
$$
y_{(s, \blam)} = (e_1 e_3 \cdots e_{2f-1}) \odot y^{\alpha(\blam)} x_\blam,
$$
and for $(\alpha, X, \mfs, v) \in \mathcal T(s, \blam)$, set
$$
v_{(\alpha, X, \mfs, v)} = \alpha( \pi(X) a(X) \odot d(\mfs) v).
$$
Combining Lemma \ref{lemma:  factorization of V 2f 0}  with the factorization
$\m \blam \mfs v \mft w = d(\mfs) v \, y^{\alpha(\blam)} x_\blam \, w^* d(\mft)^*$, we have the following factorization for the elements  of the cellular basis of $D_n(A)$.    

\begin{corollary}  For any
$(\alpha, X, \mfs, u)$ and $(\beta, Y, \mft, w) \in \mathcal T(s, \blam)$,
\begin{equation} \label{equation: symmetric factorization of basis elements of A Brauer algebra}
\begin{aligned}
&\mm {(s, \blam)} {(\alpha, X, \mfs, u)} {(\beta, Y, \mft, w) } = v_{(\alpha, X, \mfs, u)} \, y_{(s, \blam)}  \, (v_{(\beta, Y, \mft, w)})^*.
\end{aligned}
\end{equation}
\end{corollary}

 \begin{lemma}  \label{cellular axiom 3 for A Brauer algebra} \mbox{}
 \begin{enumerate}
 \item  $(y_{(s, \blam)})^* \equiv y_{(s,\blam)} \mod \overline M^{(s, \blam)}$. 
 \item
 $(\mm {(s, \blam)} {(\alpha, X, \mfs, v)} {(\beta, Y, \mft, w)})^* \equiv
 \mm {(s, \blam)} {(\beta, Y, \mft, w)}  {(\alpha, X, \mfs, v)} \mod \overline M^{(s, \blam)}$.
 \end{enumerate}
 \end{lemma}

 \begin{proof} Part (1) follows from Lemma \ref{lemma involution on general murphy type elements}  and Lemma \ref{lemma: subset of M bar s lambda}.  
 Part (2) follows from part (1) and Equation \eqref{equation: symmetric factorization of basis elements of A Brauer algebra}.
  \end{proof}

 In order to show that $\mathcal D$ is a  cellular basis of $D_n(A)$,  we examine the action of
 $D_n(A)$ on elements of the form 
 $$v_{(\alpha, X, \mfs, u)} \, y_{(s, \blam)} = \alpha(X X_0^* \odot \m \blam \mfs u {\mft^\blam} 1).
 $$

 \begin{lemma}  \label{lemma:  A Brauer cellularity 1}
  Let $S \odot T \in D_{2f}(A) \odot D_s(A)$.  For $X \in \mathbb A_f$ and  $(\mfs, v) \in \mathcal T(\blam)$,
 $$
 (S \odot T) v_{(1, X, \mfs, u)}  y_{(s, \blam)}
 $$
 is congruent modulo $\overline M^{(s, \blam)}$ to a linear combination of elements
 $$
v_{(1, X', \mfs', u')}  y_{(s, \blam)}
 $$
 with $X' \in \mathbb A_f$ and $(\mfs', u') \in \mathcal T(\blam)$.  
 \end{lemma}
 
 \begin{proof}
 Write
 $$
 Z =  (S \odot T) v_{(1, X, \mfs, u)}  y_{(s, \blam)} = 
 SX X_0^* \odot  T  \m \blam \mfs u {\mft^\blam} 1.
 $$
 If $\rank(T) < s$, then $\rank(Z) < s$, so $Z \in \bigoplus_{r < s} V^n_r$.    It follows that
 $Z$ is in the span of $M^{(r, \mu)}$ with $r < s$, and in particular $Z \in \overline M^{(s, \blam)}$. 
 Thus, we can suppose that $\rank(T) = s$,  i.e. $T \in V^s_s \cong A \wr \sym_s$.  By Proposition \ref{proposition: multiplicative property of general murphy basis elements},  $T  \m \blam \mfs u {\mft^\blam} 1$ is congruent modulo $\overline N^\blam$ to a linear combination of elements $\m \blam {\mfs'} {u'}  {\mft^\blam} 1$.    Moreover, $SX \in D_{0, 2f}(A)$,  so is a linear combination of basis elements
 $X' \in \mathbb A_f$.   Therefore $Z$ is  congruent to a linear combination of elements
  $$
 X'X_0^* \odot \m \blam {\mfs'} {u'} {\mft^\blam} 1 = v_{(1, X', \mfs', u')}  y_{(s, \blam)},
 $$
 modulo $V^{2f}_0 \odot \overline N^\blam \subseteq \overline M^{(s, \blam)}$.
  \end{proof}
 
 \begin{lemma} \label{cellular axiom 2 for A Brauer algebra}
  Let $T \in D_n(A)$  and $(\alpha, X, \mfs, u) \in \mathcal T(s, \blam)$.
   Then
 $
 T v_{(\alpha, X, \mfs, u)} y_{(s, \blam)}
 $
 is congruent modulo $\overline M^{(s, \blam)}$ to a linear combination of 
 elements $v_{(\alpha', X', \mfs', u')} y_{(s, \blam)}$.  \end{lemma}
 
 \begin{proof}
 It suffices to prove this when $T$ is one of the generators $a\power 1$,  $g_i$ or $e_i$ of $D_n(A)$, see Lemma \ref{lemma: generators of D n of A}. 
 
 {\em Case 1,}  $T = a \power 1$.    Write $j = \alpha\inv(1)$.  We have
 $$
 \begin{aligned}
 a \power 1 \, v_{(\alpha, X, \mfs, u)} y_{(s, \blam)}  &= 
 a \power 1 \alpha (X X_0^* \odot \m \blam \mfs u {\mft^\blam}  1)  \\
 &=  \alpha a \power {j} (XX_0^* \odot \m \blam \mfs u {\mft^\blam}  1)  .\\
 \end{aligned}
 $$
 Since $a \power j \in D_2f(A) \odot D_s(A)$,  the conclusion follows from Lemma \ref{lemma:  A Brauer cellularity 1}.
 
 {\em Case 2,}  $T = g_i$.  More generally, if $T = \pi \in \symn$,  then $\pi \alpha = \alpha' (\pi_1 \odot \pi_2)$,  where $\alpha'$ is a $(2f,s)$--shuffle and $\pi_1 \odot \pi_2 \in \sym_{2f}\times \sym_s$.
 Thus
 $$
 \begin{aligned}
\pi  \, v_{(\alpha, X, \mfs, u)} y_{(s, \blam)}  &= 
\pi \alpha (X X_0^* \odot \m \blam \mfs u {\mft^\blam}  1)  \\
 &=  \alpha'  (\pi_1 \odot \pi_2)(X X_0^* \odot \m \blam \mfs u {\mft^\blam}  1) ,\\
 \end{aligned}
 $$
 and again the conclusion follows from Lemma \ref{lemma:  A Brauer cellularity 1}.
 
 Next, we consider $T = e_i$.  There are two cases, depending on the values of
 $\alpha\inv(i)$  and $\alpha\inv(i+1)$.
 
 {\em Case 3,}  $T = e_i$, and both $\alpha\inv(i)$  and $\alpha\inv(i+1)$ are $\le 2f$,  or both are
 $\ge 2f + 1$.   Because $\alpha$ is a $(2f,s)$--shuffle, we have $\alpha\inv(i+1) = \alpha\inv(i) + 1$. 
 Therefore, $e_i \alpha = \alpha e_j$,  where $j = \alpha\inv(i)$.   Note that $e_j \in D_2f(A) \odot D_s(A)$.   The conclusion follows by applying Lemma \ref{lemma:  A Brauer cellularity 1}, as in the previous cases.
 
 {\em Case 4,}  $T = e_i$, and one of $\alpha\inv(i)$  and $\alpha\inv(i+1)$ is $\le 2f$ while the other is $\ge 2f + 1$.   We have
 $$
 \begin{aligned}
e_i \, v_{(\alpha, X, \mfs, u)} y_{(s, \blam)}   &= 
e_i \alpha (XX_0^* \odot \m \blam \mfs u {\mft^\blam}  1) \\
 &=  e_i \alpha (X \odot 1)(X_0^* \odot \m \blam \mfs u {\mft^\blam}  1).
 \end{aligned}
 $$
 It's easy to see that the $A$--Brauer diagram $ e_i \alpha (X \odot 1) \in D_{s, n}(A)$ has rank $s$, so using the procedure of Lemma \ref{lemma factorization of Brauer diagrams}, it can be written as
 $\alpha' (X_1 \odot X_2)$,  where  $\alpha'$  is a $(2f, s)$--shuffle,  $X_1 \in D_{0, 2f}(A)$ and $X_2 \in V^s_s$, and thus as a linear combination of terms   $\alpha' (X' \odot X_2)$,  where $X' \in \mathbb A_f$.
 Thus  $e_i \, v_{(\alpha, X, \mfs, u)} y_{(s, \blam)}$  is a linear combination of terms
 $$
 \begin{aligned}
\alpha'  (X' \odot X_2)(X_0^* \odot \m \blam \mfs v  {\mft^\blam}  1) 
&= \alpha' (X' X_0^* \odot X_2 \m \blam \mfs v  {\mft^\blam}  1)  \\
 \end{aligned}
 $$
 Now the result follows by using Proposition \ref{proposition: multiplicative property of general murphy basis elements}, and Lemma \ref{lemma: subset of M bar s lambda}.
  \end{proof}
  
\begin{corollary} \label{corollary:  cellular property 2 for basis D of A Brauer algebra}
Let $T \in D_n(A)$  and $\mm {(s, \blam)}  {(\alpha, X, \mfs, u)} {(\beta, Y, \mft, w)} \in \mathcal D$.  Then $T \mm {(s, \blam)}  {(\alpha, X, \mfs, u)} {(\beta, Y, \mft, w)}$ is congruent  modulo
$\overline M^{(s, \blam)}$
to a linear combination of basis elements $\mm {(s, \blam)}  {(\alpha', X', \mfs', u')} {(\beta, Y, \mft, w)}$, with the coefficients of the linear combination independent of $(\beta, Y, \mft, w)$.  
\end{corollary}

\begin{proof}  We have that
$
 T v_{(\alpha, X, \mfs, u)} y_{(s, \blam)}
 $
 is congruent modulo $\overline M^{(s, \blam)}$ to a linear combination of 
 elements $v_{(\alpha', X', \mfs', u')} y_{(s, \blam)}$.  Multiplying this congruence on the right by
 $(v_{(\beta, Y, \mft, w)})^*$ yields the result.
\end{proof}

\noindent
{\em Proof of Theorem \ref{theorem cellularity of A Brauer algebras}.}    By Lemma \ref{lemma:  basis of A Brauer 2}, $\mathcal D$ is an $R$--basis of $D_n(A)$.  By Lemma \ref{cellular axiom 3 for A Brauer algebra} and Corollary \ref{corollary:  cellular property 2 for basis D of A Brauer algebra},  $\mathcal D$ is a cellular basis.  Fix $(s, \blam) \in \Lambda$.  
Comparing Lemma \ref{cellular axiom 2 for A Brauer algebra} and Corollary \ref{corollary:  cellular property 2 for basis D of A Brauer algebra}, we see that
$$
M := \spn\{v_{(\alpha, X, \mfs, u)} y_{(s, \blam)} + \overline M^{(s, \blam)} :  (\alpha, X, \mfs, u) \in \mathcal T((s, \blam))  \}
$$
is an $D_n(A)$ module isomorphic to the cell module $\Delta^{(s, \blam)}$.    It follows from Corollary \ref{corollary: y alpha x lambda in span of v y alpha x lambda}
 and Lemma \ref{lemma: subset of M bar s lambda} that $y_{(s, \blam)} \in M$, and hence $\Delta^{(s, \blam)}$ is cyclic. \qed
 

\begin{thebibliography}{}

\bibitem{ariki-book}
Susumu Ariki, \emph{Representations of quantum algebras and combinatorics of
  {Y}oung tableaux}, University Lecture Series, vol.~26, American Mathematical
  Society, Providence, RI, 2002, Translated from the 2000 Japanese edition and
  revised by the author.
\MR{1911030}

\bibitem{ariki-koike}
Susumu Ariki and Kazuhiko Koike, \emph{A {H}ecke algebra of {$({\bm Z}/r{\bm
  Z})\wr{\mathfrak S}\sb n$} and construction of its irreducible representations},
  Adv. Math. \textbf{106} (1994), no.~2, 216--243.
\MR{1279219}

\bibitem{bakalov-kirillov}
Bojko Bakalov and Alexander Kirillov, Jr., \emph{Lectures on tensor categories
  and modular functors}, University Lecture Series, vol.~21, American
  Mathematical Society, Providence, RI, 2001.
\MR{1797619}

\bibitem{Brauer}
Richard Brauer, \emph{On algebras which are connected with the semisimple
  continuous groups}, Ann. of Math. (2) \textbf{38} (1937), no.~4, 857--872.
\MR{1503378}

\bibitem{chuang-meng}
Joseph Chuang and Kai~Meng Tan, \emph{Representations of wreath products of
  algebras}, Math. Proc. Cambridge Philos. Soc. \textbf{135} (2003), no.~3,
  395--411.
\MR{2018255}

\bibitem{Clifford-1937}
A.~H. Clifford, \emph{Representations induced in an invariant subgroup}, Ann.
  of Math. (2) \textbf{38} (1937), no.~3, 533--550.
\MR{1503352}

\bibitem{DJM}
Richard Dipper, Gordon James, and Andrew Mathas, \emph{Cyclotomic {$q$}-{S}chur
  algebras}, Math. Z. \textbf{229} (1998), no.~3, 385--416.
\MR{1658581}

\bibitem{Enyang1}
John Enyang, \emph{Cellular bases for the {B}rauer and
  {B}irman-{M}urakami-{W}enzl algebras}, J. Algebra \textbf{281} (2004), no.~2,
  413--449.
\MR{2098377}

\bibitem{Enyang-Goodman}
John Enyang and Frederick~M. Goodman, \emph{Cellular bases for algebras with a
  {J}ones basic construction}, in preparation (2012).

\bibitem{goodman-2008}
Frederick~M. Goodman, \emph{Cellularity of cyclotomic
  {B}irman-{W}enzl-{M}urakami algebras}, Journal of Algebra \textbf{321}
  (2009), no.~11, 3299 -- 3320, Special Issue in Honor of Gus Lehrer.
\MR{2510050}

\bibitem{Goodman-Graber1}
Frederick~M. Goodman and John Graber, \emph{Cellularity and the {J}ones basic
  construction}, Adv. in Appl. Math. \textbf{46} (2011), no.~1-4, 312--362.
\MR{2794027}

\bibitem{Goodman-Graber2}
\bysame, \emph{On cellular algebras with {J}ucys {M}urphy elements}, J. Algebra
  \textbf{330} (2011), 147--176.
\MR{2774622}

\bibitem{Graham-Lehrer-cellular}
J.~J. Graham and G.~I. Lehrer, \emph{Cellular algebras}, Invent. Math.
  \textbf{123} (1996), no.~1, 1--34.
\MR{1376244}

\bibitem{jones-hecke}
V.~F.~R. Jones, \emph{Hecke algebra representations of braid groups and link
  polynomials}, Ann. of Math. (2) \textbf{126} (1987), no.~2, 335--388.
\MR{0908150}

\bibitem{jones-planar}
Vaughan F.~R. Jones, \emph{Planar algebras, {I}}, unpublished manuscript,
  arXiv:math/9909027.

\bibitem{KX-structure}
Steffen K{\"o}nig and Changchang Xi, \emph{On the structure of cellular
  algebras}, Algebras and modules, II (Geiranger, 1996), CMS Conf. Proc.,
  vol.~24, Amer. Math. Soc., Providence, RI, 1998, pp.~365--386.
\MR{1648638}

\bibitem{KX-Morita}
\bysame, \emph{Cellular algebras: inflations and {M}orita equivalences}, J.
  London Math. Soc. (2) \textbf{60} (1999), no.~3, 700--722.
\MR{1753809}

\bibitem{macdonald-wreath}
I.~G. Macdonald, \emph{Polynomial functors and wreath products}, J. Pure Appl.
  Algebra \textbf{18} (1980), no.~2, 173--204.
\MR{0585222}

\bibitem{macdonald-book}
\bysame, \emph{Symmetric functions and {H}all polynomials}, second ed., Oxford
  Mathematical Monographs, The Clarendon Press Oxford University Press, New
  York, 1995, With contributions by A. Zelevinsky, Oxford Science Publications.
\MR{1354144}

\bibitem{Mathas-book}
Andrew Mathas, \emph{Iwahori-{H}ecke algebras and {S}chur algebras of the
  symmetric group}, University Lecture Series, vol.~15, American Mathematical
  Society, Providence, RI, 1999.
\MR{1711316}

\bibitem{Morton-Wassermann}
Hugh Morton and Antony Wassermann, \emph{A basis for the {B}irman-{W}enzl
  algebra}, Unpublished manuscript (1989, revised 2000), 1--29, placed on the
  arXiv by H.R. Morton in 2010, arXiv:1012.3116.

\bibitem{murphy-hecke95}
G.~E. Murphy, \emph{The representations of {H}ecke algebras of type {$A\sb
  n$}}, J. Algebra \textbf{173} (1995), no.~1, 97--121.
\MR{1327362}

\bibitem{G-Brauer}
M.~Parvathi and D.~Savithri, \emph{Representations of {$G$}-{B}rauer algebras},
  Southeast Asian Bull. Math. \textbf{26} (2002), no.~3, 453--468.
\MR{2047837}

\bibitem{Specht-1933}
Wilhelm Specht, \emph{Eine {V}erallgemeinerung der {P}ermutationsgruppen},
  Math. Z. \textbf{37} (1933), no.~1, 321--341.
\MR{1545398}

\bibitem{Wenzl-Brauer}
Hans Wenzl, \emph{On the structure of {B}rauer's centralizer algebras}, Ann. of
  Math. (2) \textbf{128} (1988), no.~1, 173--193.
\MR{0951511}

\bibitem{Xi-BMW}
Changchang Xi, \emph{On the quasi-heredity of {B}irman-{W}enzl algebras}, Adv.
  Math. \textbf{154} (2000), no.~2, 280--298.
\MR{1784677}

\end{thebibliography}

\def\MR#1{\href{http://www.ams.org/mathscinet-getitem?mr=#1}{MR#1}}

  \end{document}